\documentclass[a4paper,10pt,oneside,notitlepage]{article}
\usepackage{times}
\usepackage{lmodern}
\usepackage[T1]{fontenc}
\usepackage[utf8]{inputenc}
\usepackage[english]{babel}
\usepackage{amssymb}
\usepackage{amsmath}
\usepackage{mathtools} % it allows to typeset formula in display style without any commands... \begin{dcases}
\usepackage{amsthm}
\usepackage{amsfonts, mathrsfs}
\usepackage{latexsym}
\usepackage[pdftex]{color, graphicx}
\usepackage{makeidx}
\usepackage{sidecap}
\usepackage{verbatim}
\usepackage{geometry}
\usepackage{subfig}
\usepackage{tikz,float}
\usetikzlibrary{shapes,arrows,shadows}
\usepackage[displaymath,mathlines, running, pagewise]{lineno} % number of line
\usepackage[colorinlistoftodos]{todonotes}

\newtheorem{thm}{Theorem}[section]
\newtheorem{prop}[thm]{Proposition}
\newtheorem{lemma}[thm]{Lemma}

\newtheorem{definition}[thm]{Definition}
\newtheorem{remark}[thm]{Remark}

\newcommand{\leftharpoonupeq}{\mathrel{%
\raise.46ex\hbox{$\rightharpoonup $}%
\setbox0=\hbox{$\rightharpoonup$}%
\kern -.98\wd0 \lower.26ex\box0}}%

\def\rd{\mathbb{R}^d}
\def\rn{\mathbb{R}^n}
\def\rtwo{\mathbb{R}^2}
\def\rnd{ \mathbb{R}^{n\times d}}
\def\R{\mathbb{R}}
\def\o{\Omega}
\def\fe{\mathscr{F}_\varepsilon}
\def\fhom{\mathscr{F}_0}
\def\xe{\frac{x}{\varepsilon}}
\def\l{\left}
\def\r{\right}
\def\grad{\nabla}
\def\A{\mathbb{A}}

\def\div{\mathrm{div}}
\def\ue{u_\varepsilon}
\def\Cinfcpt{C^\infty_{\text{c}}}
\def\Cinfper{C^\infty_{\text{per}}}

\def\ker{\mbox{\rm Ker}}

\def\dboundary{d\mathscr{H}_y}
\def\ltwoper{L^2_{\text{per}}}
\def\honeper{H^1_{\text{per}}}

\def\gammalim{{\stackrel{\Gamma(X)-w}{\rightharpoonup}}}
\def\gammalimltwo{{\stackrel{\Gamma(L^2)-w}{\rightharpoonup}}}

\def\xetwo{\frac{x_2}{\varepsilon}}
\def\fourtra{\mathcal{F}}

\numberwithin{equation}{section}

\title{$\Gamma$-convergence of quadratic functionals with non uniformly elliptic conductivity matrices}
\author{ Lorenza D'Elia\footnote{Dipartimento di Matematica, 
 Universit\`a di Roma ``Tor Vergata'', via della ricerca scientifica 1, 00133 Roma, Italy {\tt e-mail: lorenza.delia@polito.it}}
}

\date{}

\begin{document}
\maketitle

\smallskip
\begin{abstract}
We investigate the homogenization through $\Gamma$-convergence for the $L^2(\o)$-weak topology of the conductivity functional with a zero-order term where the matrix-valued conductivity is assumed to be non strongly elliptic. Under proper assumptions, we show that the homogenized matrix $A^\ast$ is provided by the classical homogenization formula. We also give algebraic conditions for two and three dimensional $1$-periodic rank-one laminates such that the homogenization result holds.  For this class of laminates,  an explicit expression of $A^\ast$ is  provided   which is a generalization of the classical
laminate formula. We construct a two-dimensional counter-example which  shows an anomalous asymptotic behaviour of the conductivity functional. 
%This paper is devoted to the homogenization through $\Gamma$-convergence for the $L^2(\o)$-weak topology of the conductivity functional with a zero-order term where the $Y_d$-periodic, symmetric and matrix-valued conductivity $A\in L^\infty(\rd)^{d\times d}$ is assumed to be non strongly elliptic. We provide an homogenization result for such a family of functional. The homogenization matrix $A^\ast$ turns out to be given by the classical homogenization formula provided that the two-scale limit $u_0(x,y)$ of a sequence $u_\varepsilon$ with equi-bounded energy is independent of the fast variable $y\in Y_d$ and the space V %Under the independence of any two-scale limit $u_0(x,y)$ of a sequence $u_\varepsilon \in L^2(\o)$ with bounded energy from $y\in Y_d$ and the assumption that the space $V$ defined by
%                        \begin{linenomath}
%                        \begin{equation*}
                        % V:=\l\{ \int_{Y_d}A^{1/2}(y)\Phi(y)dy \hspace{0.03cm}:\hspace{0.03cm} \Phi\in L^2_{\text{per}}(Y_d; \hspace{0.03cm}\R^d) \hspace{0.2cm} \mbox{\rm with}\hspace{0.2cm}\mbox{\rm div}\l(A^{1/2}(y)\Phi(y)\r) =0 \hspace{0.2cm}\mbox{in}\hspace{0.2cm}\mathscr{D}'(\R^d) \r\}
                        %\end{equation*}
                        %\end{linenomath}
%agrees with  $\rd$. 
\end{abstract}



\section{Introduction}
In this paper, for a bounded domain $\Omega$ of $\R^d$, we study the homogenization through $\Gamma$-convergence of the conductivity energy with a zero-order term of the type
            \begin{equation}
            \label{funct}
            \fe(u) :=
            \begin{dcases}
            \int_{\o} \l\{A\l(\xe\r)\grad u\cdot\grad u + |u|^2\r\} dx,&\mbox{if $u\in H^1_0(\o)$},\\
                 &\\
            \hspace{2cm}\infty,&\mbox{if $u\in L^2(\o)\setminus H^1_0(\o)$}.
            \end{dcases}
            \end{equation}
The conductivity  $A$ is a $Y_d$-periodic, symmetric and non-negative matrix-valued function in $L^\infty(\rd)^{d\times d}$, denoted by $L^\infty_\text{per}(Y_d)^{d\times d}$, which is not strongly elliptic, \textit{i.e.}
    \begin{equation}
    \label{notstronell}
    \underset{y\in Y_d}{\mbox{\rm ess-inf}}\hspace{0.05cm} \l(\min \hspace{0.03cm}\l \{A(y)\xi\cdot\xi\hspace{0.02cm}:\hspace{0.02cm} \xi\in\rd, \, |\xi|=1 \r\}\r)\geq0.
    \end{equation} 
This condition holds true when the conductivity energy density has missing derivatives. This occurs, for example, when the quadratic form associated to $A$ is given by 
      \begin{linenomath}
      \begin{equation*}
      A\xi\cdot\xi := A'\xi'\cdot \xi'\qquad \mbox{for}\hspace{0.2cm} \xi=(\xi', \xi_d)\in\R^{d-1}\times\R,
      \end{equation*}
      \end{linenomath}
where $A'\in L^\infty_\text{per}(Y_d)^{(d-1)\times(d-1)}$ is  symmetric and non-negative matrix. 
It is known (see \textit{e.g.} \cite[ Chapters 24 and 25]{DM93}) that the strongly ellipticity of the matrix $A$, \textit{i.e.}
    \begin{equation}
    \label{strongell}
    \underset{y\in Y_d}{\mbox{\rm ess-inf}}\hspace{0.05cm} \l(\min \hspace{0.03cm}\l \{A(y)\xi\cdot\xi\hspace{0.02cm}:\hspace{0.02cm} \xi\in\rd, \, |\xi|=1 \r\}\r)>0,
    \end{equation}   
combined with the boundedness implies a compactness result of   
%assumption of uniformly boundedness of $A$, \textit{i.e.} there exist two positive constants $C_1, C_2$ such that
%     \begin{linenomath}
%     \begin{equation*}
%     C_1|\xi|^2\leq A(y)\xi\cdot\xi\leq C_2|\xi|^2 \qquad \mbox{a.e.}\hspace{0.3cm}y\in Y_d\quad\mbox{and}\quad\mbox{for any}\hspace{0.3cm}\xi\in\rd,
%     \end{equation*}
%     \end{linenomath} 
the conductivity functional 
    \begin{linenomath}
    \begin{equation*}
     u\in H^1_0(\o)\mapsto\int_{\o}A\l(\xe\r)\grad u\cdot \grad udx
    \end{equation*}
    \end{linenomath}
for the $L^2(\o)$-strong topology. The $\Gamma$-limit is given by
   \begin{linenomath}
   \begin{equation*}
   \int_{\o} A^\ast\grad u\cdot\grad u dx,
   \end{equation*}
   \end{linenomath}
where the matrix-valued function $A^\ast$ is defined by the classical homogenization formula 
    \begin{linenomath}
    \begin{equation}
    \label{stronghommat}
    A^\ast\lambda\cdot\lambda := \min \l\{\int_{Y_d} A(y)(\lambda+\grad v(y))\cdot(\lambda+\grad v(y))dy\hspace{0.03cm}:\hspace{0.03cm}v\in H^1_{\text{per}}(Y_d) \r\}.
    \end{equation}
    \end{linenomath}
The $\Gamma$-convergence for the $L^p(\o)$-strong topology, for $p>1$, for the class of integral functionals $F_\varepsilon$ of the form 
     \begin{equation}
     \label{intfunc}
     F_\varepsilon(u) = \int_{\o}f\l({x\over \varepsilon}, Du\r) dx, \qquad\mbox{for}\hspace{0.2cm} u\in W^{1,p}(\o, \R^m),
     \end{equation}
where $f: \o\times \mathbb{R}^{m\times d}\to \R$ is a Borel function, $1$-periodic in the first variable satisfying the standard growth conditions of order $p$, namely $c_1|M|^p \leq f(x, M)\leq c_2(|M|^p +1)$ for any $x\in\o$ and for any real $(m\times d)$-matrix $M$, has been widely studied and it is a classical subject (see \textit{e.g.} \cite[Chapter 12]{Brai02} and \cite[Chapter 24]{DM93}). On the contrary, the $\Gamma$-convergence of oscillating functionals for the weak topology on bounded sets of $L^p(\o)$ has been very few analysed. An example of the study of $\Gamma$-convergence for the $L^p(\o)$-weak topology can be found in the paper \cite{BCP04} where, in the context of double-porosity, the authors 
compare the $\Gamma$-limit for non-linear functionals analogous to \eqref{intfunc}  computed with respect to different topologies and in particular with respect to $L^p(\o)$-weak topology. 
\par In this paper, we investigate the $\Gamma$-convergence for the weak topology on bounded sets (a metrizable topology) of $L^2(\o)$ of the conductivity functional under  condition \eqref{notstronell}. In this case, one has no  \textit{a priori} $L^2(\o)$-bound on the sequence of gradients, which implies a loss of coerciveness of the investigated energy. 
 %With this assumption, the sequence of gradients are not \textit{a priori} bounded in $L^2(\o)$.  
%\textcolor{red}{When the strong topology of $L^2(\o)$ is replaced by the weak topology, the conductivity functional is not anymore coercive}
\noindent To overcome this difficulty, we add a quadratic zeroth-order term of the form $\displaystyle\|u\|^2_{L^2(\o)}$, so that we immediately obtain the coerciveness in the weak topology of $L^2(\o)$ of $\fe$, namely, for $u\in H^1_0(\o)$,
         \begin{linenomath}
         \begin{equation*}
         \fe(u)\geq\int_{\o} |u|^2dx.
         \end{equation*}
         \end{linenomath} 
This estimate guarantees that $\Gamma$-limit for the weak topology on bounded sets of $L^2(\o)$ is characterized by conditions  $(i)$ and $(ii)$ of the Definition \ref{defGammaconv} below (see \cite[Proposition 8.10]{DM93}), as well as, thanks to  a compactness result (see \cite[Corollary 8.12]{DM93}),  $\fe$ $\Gamma$-converges for the weak topology of $L^2(\o)$, up to subsequences, to some functional.  We will show that, under the following assumptions: 
   \begin{itemize}
        \item[\rm (H1)] any two-scale limit $u_0(x,y)$ of a sequence $\ue$ of functions in  $L^2(\o)$ with bounded energy $\fe(\ue)$ does not depend on $y$ (see \cite[Theorem 1.2]{A92});
       \item[\rm (H2)] the space $V$ defined by
                        \begin{equation*}
                         V:=\l\{ \int_{Y_d}A^{1/2}(y)\Phi(y)dy \hspace{0.03cm}:\hspace{0.03cm} \Phi\in L^2_{\text{per}}(Y_d; \hspace{0.03cm}\R^d) \hspace{0.2cm} \mbox{\rm with}\hspace{0.2cm}\mbox{\rm div}\l(A^{1/2}(y)\Phi(y)\r) =0 \hspace{0.2cm}\mbox{in}\hspace{0.2cm}\mathscr{D}'(\R^d) \r\}
                        \end{equation*}
                       agrees with the space $\rd$,
        \end{itemize} 
the $\Gamma$-limit is given by  
       \begin{equation}
       \label{homfunc}
       \fhom(u) := 
       \begin{dcases}
       \int_{\o} \l\{A^\ast \grad u\cdot\grad u + |u|^2\r\}dx,&\mbox{if $u\in H^1_0(\o)$},\\
       &\\
       \hspace{2cm}\infty,&\mbox{if $u\in L^2(\o)\setminus H^1_0(\o)$},
       \end{dcases}
       \end{equation} 
where the homogenized matrix $A^\ast$ is given through the expected  homogenization formula 
  \begin{equation}
   \label{hommat}
   A^\ast \lambda\cdot \lambda :=\inf\l\{\int_{ Y_d}A(y)(\lambda+\grad v(y))\cdot(\lambda+\grad v(y))dy\hspace{0.02cm}:\hspace{0.02cm}v\in \honeper(Y_d) \r\}.
   \end{equation}  
    %Roughly speaking, the condition (H1) states the independence of any  two-scale limit $u_0(x,y)$ {\color{red}of a sequence $\ue$ with bounded energy $\fe$} from the fast variable $y\in Y_d$, while the condition (H2) is of algebraic nature and it claims that any vector in $\R^d$ can be represented as the mean value over $Y_d$ of a divergence free periodic field in $\rd$ (see formula \eqref{setN}).\\

\par We need to make assumption (H1) since    for any sequence $\ue$ with bounded energy, \textit{i.e.} $\sup_{\varepsilon>0}\fe(\ue)<\infty$, the sequence $\grad\ue$ in $L^2(\o; \hspace{0.02cm}\R^d)$ is not bounded due to the lack of ellipticity of the matrix-valued conductivity $A(y)$.  Assumption (H2) turns out to be equivalent to the positive definiteness of the homogenized matrix (see Proposition \ref{prop:positivedefinite}).
\par In the $2$D isotropic elasticity setting of \cite{BF19}, 
the authors make use of similar conditions as (H1) and (H2) in the proof of the main results (see \cite[Theorems 3.3 and 3.4]{BF19}). They investigate the limit in the sense of $\Gamma$-convergence for the $L^2(\o)$-weak topology of the elasticity functional with a zeroth-order term  in the case of two-phase  isotropic laminate materials where the phase $1$ is very strongly elliptic, while the phase $2$ is only strongly elliptic. The strong ellipticity of the effective tensor is preserved  through a homogenization process %using the $\Gamma$-convergence 
expect in the case when the volume fraction of each phase is $1/2$, as first evidenced by Guti\'{e}rrez  \cite{G98}. Indeed,  Guti\'{e}rrez   has provided   two and three dimensional examples of $1$-periodic rank-one laminates such that the homogenized tensor induced by a homogenization process, labelled $1^\ast$-convergence, is not strongly elliptic.   
These examples  have been revisited by means of a homogenization process using  $\Gamma$-convergence in the two-dimensional case of \cite{BF15} and in the three-dimensional case of \cite{BPM17}.
\par  In the present scalar case, we enlighten assumptions (H1) and (H2) which are the key ingredients to obtain the general $\Gamma$-convergence result Theorem \ref{genthm}. Using  Nguetseng-Allaire \cite{A92, N89} two-scale convergence, we prove that for any dimension $d\geq 2$, the $\Gamma$-limit $\fhom$ \eqref{homfunc} for the weak topology of $L^2(\o)$ actually agrees with the one obtained for the $L^2(\o)$-strong topology under uniformly ellipticity  \eqref{strongell}, replacing the minimum in \eqref{stronghommat} by the infimum in \eqref{hommat}. Assumption (H2) implies the coerciveness of the functional $\fhom$ showing that its domain is $H^1_0(\o)$ and that the homogenized matrix $A^\ast$ is positive definite. More precisely, the positive definiteness of $A^\ast$ turns out to be equivalent to assumption (H2) (see Proposition \ref{prop:positivedefinite}). %The coerciveness of the functional $\fhom$ follows from assumption (H2) which shows that the domain of $\fhom$ is $H^1_0(\o)$. Moreover, assumption (H2) turns out to be equivalent to positive definiteness of the homogenized conductivity $A^\ast$ (see Theorem \ref{genthm}).
 We also provide two and three dimensional  $1$-periodic rank-one laminates  which satisfy assumptions (H1) and (H2) (see Proposition \ref{propd2} for the two-dimensional case and Proposition  \ref{thmcasematrixdege} for  the three-dimensional case). Thanks to Theorem \ref{genthm}, the corresponding homogenized matrix $A^\ast$ is positive definite.  For this class of laminates, an alternative and independent proof of positive definiteness of $A^\ast$ is performed  using an explicit expression of $A^\ast$ (see Proposition \ref{prop:Appendix}). This expression generalizes the classical laminate formula  for non-degenerate phases (see \cite{T85} and also \cite[Lemma 1.3.32]{A12}, \cite{B94}) to the case of two-phase rank-one laminates  with degenerate and anisotropic phases.

\par The lack of  assumption (H1) may induce a  degenerate asymptotic behaviour of the functional $\fe$ \eqref{funct}. We provide a two-dimensional rank-one laminate with two degenerate phases for which the functional $\fe$ does $\Gamma$-converge for the $L^2(\o)$-weak topology to a functional $\mathscr{F}$ which differs from the one given by \eqref{homfunc} (see Proposition \ref{prop1}). In this example, any two-scale limit $u_0(x,y)$ of  a sequence with bounded energy $\fe(\ue)$, depends on the variable $y$. Moreover, we give two quite different expressions of the $\Gamma$-limit $\mathscr{F}$ which seem to be original up to the best of our knowledge. The energy density of the first expression is written with Fourier transform of the target function. The second expression appears as a non-local functional due to the presence of a convolution term. However, we do not know if the $\Gamma$-limit $\mathscr{F}$ is a Dirichlet form in the sense of Beurling-Deny \cite{BD58}, since the Markovian property is not stable by the $L^2(\o)$-weak topology (see Remark \ref{rmk:BDrepresentation}).

%The first one is the most natural since it comes from the direct computation of $\Gamma$-limit but it is also trickier since the limit energy density is expressed by means of Fourier transform of the target function. The other one seems more interesting and original up to the best our knowledge. Indeed, the $\Gamma$-limit assumes a form which is close to that of the initial energy $\fe$ but, in accordance with the Beurling-Deny representation formula (see \cite{BD58}), it seems a non-local type functional due to the presence of the convolution term. However, up to now, it is not clear if the second expression of the $\Gamma$-limit satisfies the Markovian property due to the weak convergence (see Remark \ref{rmk:BDrepresentation}).\smallskip
\smallskip

The paper is organized as follows. In Section 2, we prove a general $\Gamma$-convergence result (see Theorem \ref{genthm}) for the functional $\fe$ \eqref{funct} with any non-uniformly elliptic matrix-valued function $A$, under assumptions (H1) and (H2). In Section 3 we illustrate the general result of Section 2 by periodic two-phase rank-one laminates with two (possibly) degenerate and anisotropic phases in dimension two and three. We provide algebraic conditions so that assumptions (H1) and (H2)  are satisfied (see Propositions \ref{propd2} and \ref{thmcasematrixdege}). In Section $4$ we exhibit a two-dimensional counter-example where assumption (H1) fails, which leads us to a degenerate $\Gamma$-limit $\mathscr{F}$ involving a convolution term (see Proposition \ref{prop1}). Finally, in the Appendix we give an explicit formula for the homogenized matrix $A^\ast$ for any two-phase rank-one laminates with (possibly) degenerate phases. We also provide  an alternative proof of the positive definiteness of $A^\ast$ using an explicit expression of $A^\ast$ for the class of two-phase rank-one laminates introduced in Section 3 (see Proposition \ref{prop:Appendix}).

\subsection*{Notation}
   \begin{itemize}
   \item   For $i=1,\dots,d$, $e_i$ denotes the $i$-th vector of the canonical basis in $\rd$;
   \item $I_d$ denotes the unit matrix of $\R^{d\times d}$;
   %\item $A\cdot B$ is the Frobenius inner product between two elements of $A, B\in\mathbb{R}^{n\times d}$, that is $A\cdot B :=\tr (A^TB)$;
   \item $\honeper(Y_d; \rn)$ (resp. $\ltwoper(Y_d; \rn)$, $\Cinfper(Y_d; \rn)$) is the space of those functions in $H^1_{\text{loc}}(\rd; \rn)$ (resp. $L^2_{\text{loc}}(\rd; \rn)$,  $C^\infty_{\text{loc}}(\rd; \rn)$) that are $Y_d$-periodic;
   %\item  \textcolor{red}{serve davvero?' o lo metto all'inizio della Section3?}For any subset $Z\in Y_d$ we denote by $Z^\#$ its representative in $\rd$, i.e. the open \textquotedblleft periodic\textquotedblright\, set 
    %     \begin{equation*}
     %    Z^\# =\widehat{\bigcup_{k\in\mathbb{Z}^d}(k+\overline{Z})};
      %   \end{equation*}
      
   \item Throughout, the variable $x$ will refer to running point in a bounded open domain $\o\subset\rd$, while the variable $y$ will refer to a running point in $Y_d$ (or $k+Y_d$, $k\in\mathbb{Z}^d$);
   \item We write 
        \begin{linenomath}
        \begin{equation*}
         \ue\leftharpoonupeq u_0
        \end{equation*}
        \end{linenomath}
   with $\ue\in L^2(\o)$ and $u_0\in L^2(\o\times Y_d)$ if $\ue$ two-scale converges to $u^0$ in the sense of Nguetseng-Allaire (see  \cite{A92, N89}) 
   \item $\fourtra_1$ and $\fourtra_2$ denote the Fourier transform defined on $L^1(\mathbb{R})$ and $L^2(\mathbb{R})$ respectively. For $f\in L^1(\R)\cap L^2(\R)$, the Fourier transform $\fourtra_1$ of $f$ is defined by 
           \begin{linenomath}
           \begin{equation*}
           \fourtra_1(f)(\lambda) := \int_{\R}e^{-2\pi i\lambda x}f(x)dx.
           \end{equation*}
           \end{linenomath} 
   \end{itemize} 
\begin{definition}
\label{defGammaconv}
Let $X$ be a reflexive and separable Banach space endowed with the weak topology $\sigma(X, X')$, and let $\fe: X\to\R$ be a $\varepsilon$-indexed sequence of functionals. The sequence $\fe$ $\Gamma$-converges to the functional $\fhom:X\to\R$ for the weak topology of $X$, and we write $\fe\gammalim\fhom$, if for any $u\in X$, 
   \begin{itemize}
   \item[i) ] $\displaystyle\forall \ue\rightharpoonup u$, $\displaystyle\fhom(u)\leq \liminf_{\varepsilon\to 0} \fe(\ue)$, 
   \item[ii) ] $\displaystyle\exists\overline{u}_{\varepsilon}\rightharpoonup u$ such that $\displaystyle\lim_{\varepsilon\to 0}\fe(\overline{u}_{\varepsilon})=\fhom(u)$. 
   \end{itemize}
Such a sequence $\overline{u}_{\varepsilon}$ is called a recovery sequence.
\end{definition}
Recall that the weak topology of $L^2(\o)$ is metrizable on bounded sets, \textit{i.e.} there exists a metric $d$ on $L^2(\o)$ such that on every norm bounded subset $B$ of $L^2(\o)$ the weak topology coincides with the topology induced on $B$ by the metric $d$ (see \textit{e.g.} \cite[Proposition 8.7]{DM93}).

\section{A preliminary general $\boldsymbol{\Gamma}$-result}
In this section, we will prove the main result of this paper. As previously announced, up to a subsequence,  the sequence of functionals $\fe$, given by \eqref{funct}  with non-uniformly elliptic matrix-valued conductivity $A(y)$, $\Gamma$-converges for the weak topology on bounded sets of $L^2(\o)$ to some functional. Our aim is to show that $\Gamma$-limit is exactly $\fhom$ when $u\in H^1_0(\o)$.   
     \begin{thm}
     \label{genthm}
     Let $\fe$ be functionals given by \eqref{funct} with $A(y)$ a $Y_d$-periodic, symmetric, non-negative matrix-valued function in $L^\infty(\R^d)^{d\times d}$ satisfying \eqref{notstronell}. Assume the following assumptions 
        \begin{itemize}
             \item[\rm (H1)] any two-scale limit $u_0(x,y)$ of a sequence $\ue$ of functions in  $L^2(\o)$ with bounded energy $\fe(\ue)$ does not depend on $y$;
            \item[\rm (H2)] the space $V$ defined by
                         \begin{linenomath}                \begin{equation}
                         \label{setN}
                              V:=\l\{ \int_{Y_d}A^{1/2}(y)\Phi(y)dy \hspace{0.03cm}:\hspace{0.03cm} \Phi\in L^2_{\text{per}}(Y_d; \hspace{0.03cm}\R^d) \hspace{0.2cm} \mbox{\rm with}\hspace{0.2cm}\mbox{\rm div}\l(A^{1/2}(y)\Phi(y)\r) =0 \hspace{0.2cm}\mbox{in}\hspace{0.2cm}\mathscr{D}'(\R^d) \r\}
                             \end{equation}
                             \end{linenomath}
                             agrees with the space $\rd$.
             \end{itemize}
     
\noindent Then, $\fe$ $\Gamma$-converges  for the weak topology of $L^2(\o)$ to $\fhom$, \textit{i.e.} 
                \begin{linenomath}
                \begin{equation*}
                %\label{mainconvergence}
                \fe\gammalimltwo\fhom,
                \end{equation*}
                \end{linenomath}    
         where  $\fhom $ is defined by \eqref{homfunc} and $A^\ast$ is given by \eqref{hommat}. 
     \end{thm}

\begin{proof}
 We split the proof into two steps which are an adaptation of \cite[Theorem 3.3]{BF19} using the sole assumptions (H1) and (H2) in the general setting of conductivity.

\smallskip
\noindent\textit{Step 1} - $\Gamma$-$\liminf$ inequality.\\
%We first prove the $\Gamma$-$\liminf$ inequality. 
Consider a sequence $\{\ue \}_\varepsilon$ converging weakly in $L^2(\o)$ to $u\in L^2(\o)$.  We want to prove that 
      \begin{equation}
      \label{Gammaliminf}
      \liminf_{\varepsilon\to 0 } \fe (\ue)\geq \fhom(u).
      \end{equation}
If the lower limit is $\infty$ then \eqref{Gammaliminf} is trivial. Up to a subsequence, still indexed by $\varepsilon$, we may assume that $\displaystyle\liminf\fe(\ue)$ is a limit and we can also assume henceforth that, for some $0<C<\infty$,
        \begin{equation}
        \label{liminf}
        \fe(\ue)\leq C. 
        \end{equation}
As $\ue$ is  bounded in $L^2(\o)$, there exists a subsequence, still indexed by $\varepsilon$, which two-scale converges to a function $u_0(x,y)\in L^2(\o\times Y_d)$ (see \textit{e.g.} \cite[Theorem 1.2]{A92}). In other words,
    \begin{equation}
    \label{twoscaue}
    \ue \leftharpoonupeq u_0.
    \end{equation}
Assumption (H1) ensures that 
   \begin{equation}
   \label{uoindy}
   u_0(x,y) = u(x) \quad\mbox{is independent of} \hspace{0.15cm}y,
   \end{equation}
where, according to the link between two-scale and weak $L^2(\o)$-convergences (see \cite[Proposition 1.6]{A92}),  $u$ is the weak limit of $\ue$, \textit{i.e.}
     \begin{linenomath} 
    \begin{equation*}
    \ue\rightharpoonup u \quad\mbox{weakly in}\hspace{0.1cm} L^2(\o).
    \end{equation*} 
    \end{linenomath}
Since all the components of the matrix $A(y)$ are bounded and  $A(y)$ is non-negative as a quadratic form, in view of \eqref{liminf}, for another subsequence (not relabeled), we have
     \begin{linenomath}  
     \begin{equation*}
     A\l(\xe\r)\grad\ue \leftharpoonupeq \sigma_0(x,y) \qquad\text{with}\hspace{0.2cm} \sigma_0\in L^2(\o\times Y_d; \rd) ,
     \end{equation*}
     \end{linenomath}
and also 
       \begin{equation}
       \label{twoscaledivfree}
       A^{1/2}\l(\xe\r)\grad\ue \leftharpoonupeq \Theta_0(x,y) \qquad\text{with}\hspace{0.2cm} \Theta_0\in L^2(\o\times Y_d; \rd).
      \end{equation}
In particular 
       \begin{equation}
       \label{twoscale0}
       \varepsilon A\l(\xe\r)\grad\ue \leftharpoonupeq 0.
       \end{equation}
Consider $\Phi\in\ltwoper(Y_d; \rd)$ such that 
      \begin{equation}
      \label{divzero}
      \div\l(A^{1/2}(y)\Phi(y)\r) = 0 \qquad\mbox{in}\hspace{0.2cm}\mathscr{D}'(\R^d),
      \end{equation}
or equivalently, 
         \begin{linenomath}
         \begin{equation}
         \notag%\label{expdivfree}
         \int_{Y_d} A^{1/2}(y)\Phi(y)\cdot\grad\psi(y)dy=0 \qquad\forall \psi\in\honeper(Y_d).
         \end{equation}
         \end{linenomath}
Take also $\varphi\in C^{\infty}(\overline{\o})$. Since $\ue\in H^1_0(\o)$ and in view of \eqref{divzero}, an integration by parts yields 
      \begin{linenomath}
      \begin{align*}
      \int_{\o} A^{1/2}\l(\xe\r)&\grad\ue\cdot\Phi\l(\xe\r)\varphi(x)dx = %-\int_{\o}\ue\div(A^{1/2}(y)\Phi(y)\varphi(x))\l(x,\xe\r)dx\\
      %&=-\frac{1}{\varepsilon}\int_{\o}\ue\div_y(A^{1/2}(y)\Phi(y))\l(x,\xe\r)\varphi(x)dx - \int_{\o}\ue A^{1/2}\l(\xe\r)\Phi\l(\xe\r)\cdot\grad\varphi(x)dx\\
       - \int_{\o}\ue A^{1/2}\l(\xe\r)\Phi\l(\xe\r)\cdot\grad\varphi(x)dx.
      \end{align*}
      \end{linenomath}
By using \cite[Lemma 5.7]{A92},  $A^{1/2}(y)\Phi(y)\cdot\grad \varphi(x)$ is an admissible test function for the two-scale convergence. Then, we can pass to the two-scale limit in the previous expression with the help of the convergences 
\eqref{twoscaue} and \eqref{twoscaledivfree} along with \eqref{uoindy}, and we obtain
      \begin{equation}
      \label{f2}
      \int_{\o\times Y_d} \Theta_0(x,y)\cdot\Phi(y)\varphi(x)dxdy = -\int_{\o\times Y_d} u(x)A^{1/2}(y)\Phi(y)\cdot\grad\varphi(x)dxdy. 
      \end{equation}
We prove that the target function $u$ is in $H^1(\Omega)$. Setting
     \begin{equation}
      \label{N}
      N:=\int_{Y_d} A^{1/2}(y)\Phi(y)dy,
     \end{equation}
and varying $\varphi$ in $\Cinfcpt(\o)$, the equality \eqref{f2} reads as 
      \begin{linenomath}
      \begin{equation*}
      \label{formulaN}
      \int_{\o\times Y_d} \Theta_0(x,y)\cdot \Phi(y)\varphi(x)dxdy =-\int_{\o} u(x)N\cdot \grad\varphi(x)dx
      \end{equation*}
      \end{linenomath}      
Since the integral in the left-hand side is bounded by a constant times $\displaystyle\|\varphi\|_{L^2(\o)}$, the right-hand side is a linear and continuous map in $\displaystyle\varphi\in L^2(\o)$. By the Riesz representation theorem, there exists $g\in L^2(\o)$ such that, for any $\varphi\in\Cinfcpt(\o)$,
        \begin{linenomath}
        \begin{equation*}
        \int_{\o} u(x)N\cdot \grad\varphi(x)dx  = \int_{\o} g(x)\varphi(x)dx,
        \end{equation*}
        \end{linenomath}
which implies that
      \begin{equation}
       \label{f4}
       N\cdot\grad u\in L^2(\o).
       \end{equation} 
In view of assumption (H2), $N$ is an arbitrary vector in $\rd$ so that we infer from \eqref{f4} that 
    \begin{equation}
    \label{uH1}
    u\in H^1(\o).
    \end{equation}
This combined with equality  \eqref{f2} leads us to 
    \begin{equation}
    \label{f3}
    \int_{\o\times Y_d} \Theta_0(x,y)\cdot\Phi(y)\varphi(x)dxdy = \int_{\o\times Y_d}A^{1/2}(y)\grad 
    u(x)\cdot\Phi(y)\varphi(x)dxdy. 
    \end{equation}
By density, the last equality holds if the test functions $\Phi(y)\varphi(x)$ are replaced by the set of $\psi(x,y)\in L^2(\o; \ltwoper(Y_d;\rd))$ such that 
          \begin{linenomath}
          \begin{equation*}
          \div_y\l(A^{1/2}(y)\psi(x,y)\r)=0 \qquad\mbox{in}\hspace{0.2cm}\mathscr{D}'(\R^d),
          \end{equation*} 
          \end{linenomath} 
or equivalently,
      \begin{linenomath}
     \begin{equation*}
     \int_{\o\times Y_d}\psi(x,y)\cdot A^{1/2}(y)\grad_yv(x,y)dxdy=0 \qquad\forall v\in L^2(\o;\honeper(Y_d)).
     \end{equation*}
     \end{linenomath}
The $L^2(\o; \ltwoper(Y_d;\rd))$-orthogonal to that set is the $L^2$-closure of 
      \begin{linenomath}
      \begin{equation*}
      \mathscr{K} :=\l\{A^{1/2}(y)\grad_yv(x,y)\hspace{0.02cm}:\hspace{0.02cm} v\in L^2(\o; \honeper(Y_d))\r\}.
      \end{equation*}
      \end{linenomath}
Thus, the equality \eqref{f3} yields
        \begin{linenomath}
        \begin{equation*}
        \Theta_0(x,y) = A^{1/2}(y)\grad u(x)+S(x,y)
        \end{equation*} 
        \end{linenomath}
for some $S$ in the closure of $\mathscr{K}$, \textit{i.e.} there exists a sequence $v_n\in L^2(\o;\honeper(Y_d))$ such that 
      \begin{linenomath}
      \begin{equation*}
      A^{1/2}(y)\grad_yv_n(x,y) \to S(x,y)\qquad\text{strongly in}\quad L^2(\o;\ltwoper(Y_d;\rd)).
      \end{equation*}
      \end{linenomath}
\par Due to the lower semi-continuity property of two-scale convergence (see \cite[Proposition 1.6]{A92}), we get 
       \begin{linenomath}
       \begin{align*}
       \liminf_{\varepsilon\to 0} \|A^{1/2}(x/\varepsilon)\grad\ue\|^2_{L^2(\o;\rd)}&\geq \|\Theta_0\|^2_{L^2(\o\times Y_d;\rd)}\\
       &=\lim_n\l\|A^{1/2}(y)\l(\grad_xu(x)+\grad_yv_n\r)\r\|^2_{L^2(\o\times Y_d;\rd)}.
       \end{align*}
       \end{linenomath}
Then, by the weak $L^2$-lower semi-continuity of $\|\ue\|_{L^2(\o)}$, we have 
      \begin{linenomath}
      \begin{align*}
      \liminf_{\varepsilon\to 0}\fe(\ue)&\geq \lim_{n}\int_{\o\times Y_d}A(y)(\grad_xu(x)+\grad_yv_n(x,y))\cdot(\grad_xu(x)+\grad_yv_n(x,y))dxdy\\
      &\quad + \int_{\o}|u|^2dx\\
      &\geq\int_{\o}\inf\l\{\int_{ Y_d}A(y)(\grad_xu(x)+\grad_yv(y))\cdot(\grad_xu(x)+\grad_yv(y))dy\hspace{0.02cm}:\hspace{0.02cm}v\in \honeper(Y_d) \r\} dx\\
      &\quad + \int_{\o}|u|^2dx.
      \end{align*}
      \end{linenomath}
Recalling the definition \eqref{hommat},
%Defining the homogenized matrix as 
 %     \begin{equation*}
%      A^\ast \eta\cdot\eta =\inf\l\{\int_{\o\times Y_d}A(y)(\eta+\grad_yv(y))\cdot(\eta+\grad_yv(y))dxdy\hspace{0.02cm}:\hspace{0.02cm}v\in \honeper(Y_d)   \r\}, 
%      \end{equation*}
we immediately conclude that
       \begin{linenomath}
       \begin{equation*}
        \liminf_{\varepsilon\to 0}\fe(\ue)\geq \int_{\o}\l\{A^\ast\grad u\cdot\grad u + |u|^2\r\}dx,
       \end{equation*}
       \end{linenomath}
provided that $u\in H^1_0(\o)$.%$u\in H^1(\o)$, hence, \textit{a fortiori} provided that $u\in H^1_0(\o)$.
\par It remains to prove that  the target function $u$ is actually in $H^1_0(\o)$, giving a complete characterization of $\Gamma$-limit. To this end, take $x_0\in\partial\o$ a Lebesgue point for $\displaystyle u\lfloor\partial\o$ and for $\nu (x_0)$, the exterior normal to $\o$ at point $x_0$.  Thanks to \eqref{uH1}, we know that $u\in H^1(\o)$, hence, after an integration by parts of the right-hand side of \eqref{f2},  we obtain, for $\varphi\in C^\infty(\overline{\o})$,
    \begin{linenomath}
    \begin{align}
    \label{uH10}
    \int_{\o\times Y_d}\Theta_0(x,y)\cdot\Phi(y)\varphi(x)dxdy = \int_{\o}N\cdot\grad u(x)\varphi(x)dx - \int_{\partial\o}N\cdot \nu(x)u(x)\varphi(x)d\mathscr{H},
    \end{align}
    \end{linenomath}
where $N$ is given by \eqref{N}. Varying $\varphi$  in $\Cinfcpt(\o)$, the first two integrals in \eqref{uH10} are equal and  bounded by a constant times $\|\varphi\|_{L^2(\o)}$. % This means that $\Theta_0(x,y)\cdot \Phi(y) = N\cdot\grad u$ a.e. x\in\o (see Buttazzo " Variational analysis in Sobolev spaces..." pag 239)
 It follows that,  for any $\varphi\in C^\infty(\overline{\o})$,
    \begin{linenomath}
    \begin{equation*}
    \int_{\partial\o}N\cdot \nu(x)u(x)\varphi(x)d\mathscr{H} =0,
    \end{equation*}
    \end{linenomath}
which leads to $N\cdot \nu(x) u(x)=0$ $\mathscr{H}$-a.e. on $\partial\o$. %(see Lemma6.2.1 Buttazzo)
 Since $x_0$ is a Lebesgue point, we have
      \begin{linenomath}
      \begin{equation}
      \label{finform}
      N\cdot \nu(x_0) u(x_0)=0.
      \end{equation}
      \end{linenomath}
In view of assumption (H2) and the arbitrariness of $N$, we can choose $N$ such that $N=\nu(x_0)$ so that from \eqref{finform} we get $u(x_0)=0$. Hence, 
     \begin{linenomath}
     \begin{equation*}
     u\in H^1_0(\o).
     \end{equation*}
     \end{linenomath}
This concludes the proof of the $\Gamma$-$\liminf$ inequality.

\smallskip
\noindent\textit{Step 2} - $\Gamma$-$\limsup$ inequality.\\
%As far as the $\Gamma$-$\limsup$ inequality is concerned, 
We use the same arguments of \cite[Theorem 2.4]{BPM17} which can easily extend to the conductivity setting. We just give an idea of the proof, which is based on a perturbation argument. %Under the sole assumption that the matrix $A$ is non-negative, it is possible to prove that the functional $\fe$ $\Gamma$-converges in the strong topology of $L^2(\o)$ to the functional $\fhom$. The underlying idea is to
For $\delta>0$, let $A_\delta$ be the perturbed matrix of $\R^{d\times d}$ defined by 
     \begin{linenomath}
     \begin{equation*}
     A_\delta:=A+\delta I_d,
     \end{equation*}
     \end{linenomath}
where $I_d$ is the unit matrix of $\R^{d\times d}$. Since the matrix $A$ is non-negative, $A_\delta$ turns out to be positive definite, hence, the  functional $\fe^\delta$, defined by \eqref{funct} with $A_\delta$ in place of $A$,
%      \begin{equation*}
%      \fe^\delta(u) = \begin{dcases}
%      \int_{\o} \l\{A_\delta\l(\xe\r)\grad u\cdot\grad u + |u|^2\r\}dx&\mbox{if}\hspace{0.3cm} u\in H^1_0(\o),\\
%      \hspace{2cm}\infty &\mbox{if}\hspace{0.3cm} u\in L^2(\o)\setminus H^1_0(\o)
%      \end{dcases}
%      \end{equation*}
$\Gamma$-converges to the functional $\mathscr{F}^\delta$ given by
      \begin{linenomath}
      \begin{equation*}
      \mathscr{F}^\delta(u) := \begin{dcases}
      \int_{\o} \l\{A^\ast_\delta\grad u\cdot\grad u + |u|^2\r\}dx,&\mbox{if}\hspace{0.3cm} u\in H^1_0(\o),\\
      \hspace{2cm}\infty, &\mbox{if}\hspace{0.3cm} u\in L^2(\o)\setminus H^1_0(\o),
      \end{dcases}
      \end{equation*}
      \end{linenomath}
for the strong topology of $L^2(\o)$ (see \textit{e.g.} \cite[Corollary 24.5]{DM93}).
Thanks to the compactness result of $\Gamma$-convergence (see \textit{e.g.} \cite[Proposition 1.42]{Brai02}), there exists a subsequence $\varepsilon_j$ such that $\mathscr{F}_{\varepsilon_j}$ $\Gamma$-converges for the $L^2(\o)$-strong topology to some functional $F^0$. Let $u\in H^1_0(\o)$ and let $u_{\varepsilon_j}$ be a recovery sequence for $\mathscr{F}_{\varepsilon_j}$ which  converges to $u$ for the $H^1(\o)$-weak topology on bounded sets. Since $\displaystyle\mathscr{F}_{\varepsilon_j}\leq \mathscr{F}_{\varepsilon_j}^\delta$ and since $u_{\varepsilon_j}$ belongs to some bounded set of $H^1(\Omega)$, from \cite[Propositions 6.7 and 8.10]{DM93}  we deduce that
    \begin{linenomath}
    \begin{align*}
    F^0(u)&\leq \mathscr{F}^\delta(u)\\
    &\leq \liminf_{\varepsilon_j\to 0}  \int_{\o}\l\{A_\delta\l(\frac{x}{\varepsilon_j}\r)\grad u_{\varepsilon_j}\cdot \grad u_{\varepsilon_j} +|u_{\varepsilon_j}|^2\r\}dx\\
    &\leq \liminf_{\varepsilon_j\to 0}  \int_{\o}\l\{A\l(\frac{x}{\varepsilon_j}\r)\grad u_{\varepsilon_j}\cdot \grad u_{\varepsilon_j} +|u_{\varepsilon_j}|^2\r\}dx + O(\delta)\\
    &= F^0(u) + O(\delta).
    \end{align*}
    \end{linenomath}
It follows that $\mathscr{F}^\delta$ converges to $F^0$ as $\delta\to 0$. Then, the $\Gamma$-limit $F^0$ of $\mathscr{F}_{\varepsilon_j}$ is independent on the subsequence $\varepsilon_j$. Repeating the same arguments, any subsequence of $\fe$ has a further subsequence which $\Gamma$-converges for the strong topology of $L^2(\o)$ to $F^0=\lim_{\delta\to 0}\mathscr{F}^\delta$. Thanks to the Urysohn property  (see \textit{e.g.} \cite[Proposition 1.44]{Brai02}),  the whole sequence $\fe$ $\Gamma$-converges to the functional $F^0$ for the strong topology of $L^2(\o)$. On the other hand, in light of the definition \eqref{hommat} of $A^\ast$, we get that $A^\ast_\delta$ converges to $A^\ast$ as $\delta\to 0$, \textit{i.e.}
    \begin{equation}
    \label{AdeltatendA}
    \lim_{\delta\to 0}A^\ast_\delta=A^\ast.
    \end{equation}
Thanks to the Lebesgue dominated convergence theorem and in view of \eqref{AdeltatendA}, we get that $F^0 = \lim_{\delta\to 0}\mathscr{F}^\delta$ is exactly $\fhom$ given by \eqref{homfunc}. Therefore, $\fe$ $\Gamma$-converges to $\fhom$ for the $L^2(\o)$-strong topology. 
\par Now, let us show that $\fe$ $\Gamma$-converges to $\fhom$ for the weak topology of $L^2(\o)$. Recall that the $L^2(\o)$-weak topology is metrizable on the closed ball of $L^2(\o)$.  Fix $n\in\mathbb{N}$ and let $d_{B_n}$ be any metric inducing the $L^2(\o)$-weak topology on the ball $B_n$ centered on $0$ and of radius $n$. Let $u\in H^1_0(\o)$ and let $\overline{u}_{\varepsilon}$ be a recovery sequence for $\fe$ for the $L^2(\o)$-strong topology. Since the topology induced  by the metric $d_{B_n}$ on $B_n$ is weaker than the $L^2(\o)$-strong topology,  $\overline{u}_{\varepsilon}$ is also a recovery sequence for $\fe$ for the $L^2(\o)$-weak topology on $B_n$. Hence,
    \begin{linenomath} 
    \begin{equation*}
    \lim_{\varepsilon\to 0}\fe(\overline{u}_{\varepsilon}) =\fhom(u),
    \end{equation*}
    \end{linenomath}  
which proves the $\Gamma$-$\limsup$ inequality in $B_n$. Finally, since any sequence converging weakly in $L^2(\o)$ belongs to some ball $B_n\subset L^2(\o)$, as well as its limit, it follows that the $\Gamma$-$\limsup$ inequality holds true for $\fe$ for $L^2(\o)$-weak topology, which concludes the proof.
\end{proof}

The next proposition provides a characterization of Assumption {\rm (H2)} in terms of homogenized matrix $A^\ast$.
    \begin{prop}
    \label{prop:positivedefinite}
    Assumption {\rm (H2)} is equivalent to the positive definiteness of $A^\ast$, or equivalently,
                \begin{equation}
                \notag%\label{kerA=Vperp}
                \ker(A^\ast)=V^\perp.
                \end{equation}
    \end{prop}
\begin{proof}
Consider $\displaystyle \lambda\in \ker(A^\ast)$.
%Assume assumption (H2). We want to prove that $A^\ast$ is  a positive definite matrix in $\R^{d\times d}$. To this end, assume also that $A^\ast\lambda\cdot\lambda=0$. %This combined with the definition \eqref{hommat} of $A^\ast$, yields to 
  %  \begin{linenomath}
  %  \begin{equation*}
   % A^\ast\lambda\cdot\lambda = \inf\l\{ \int_{ Y_d}A(y)(\lambda+\grad v(y))\cdot(\lambda+\grad v(y))dy\hspace{0.02cm}:\hspace{0.02cm}v\in \honeper(Y_d)   \r\} =0.
   % \end{equation*}
   % \end{linenomath} 
Define 
   \begin{linenomath}
   \begin{equation*}
   H^1_\lambda(Y_d) := \l\{u\in H^1_{\rm loc}(\rd)\hspace{0.03cm}:\hspace{0.03cm} \grad u \hspace{0,2cm}\mbox{is} \hspace{0.2cm}Y_d\mbox{-periodic and}\hspace{0.2cm} \int_{Y_d}\grad u(y)dy=\lambda       \r\}.
   \end{equation*}
   \end{linenomath}
Recall that $u\in H^1_\lambda(Y_d)$ if and only if there exists $v\in\honeper(Y_d)$ such that $u(y)= v(y) + \lambda\cdot y $ (see \textit{e.g.} \cite[Lemma 25.2]{DM93}). Since $A^\ast$ is non-negative and symmetric, from  \eqref{hommat} it follows that 
  \begin{linenomath}
  \begin{align*}
  0=A^\ast\lambda\cdot\lambda %&= \inf\l\{ \int_{ Y_d}A(y)(\lambda+\grad v(y))\cdot(\lambda+\grad v(y))dy\hspace{0.02cm}:\hspace{0.02cm}v\in \honeper(Y_d)   \r\}\notag\\
  &=\inf\l\{ \int_{ Y_d}A(y)\grad u(y)\cdot\grad u(y)dy\hspace{0.02cm}:\hspace{0.02cm}u\in H^1_\lambda(Y_d)   \r\}.\label{newhommat}
  \end{align*}
  \end{linenomath}
Then, there exists a sequence $u_n$ of functions in $H^1_\lambda(Y_d)$ such that 
    \begin{linenomath}
    \begin{equation*}
    \lim_{n\to\infty}\int_{Y_d} A(y)\grad u_n(y)\cdot\grad u_n(y) dy = 0,
    \end{equation*}
    \end{linenomath}
which implies that 
    \begin{linenomath}
    \begin{equation}
    \label{Astrongconver}
    A^{1/2}\grad u_n \to 0\qquad \mbox{strongly in}\hspace{0.2cm} L^2(Y_d; \hspace{0.02cm} \R^d).
    \end{equation}
    \end{linenomath}
Now, take $\Phi\in L^2_{\rm per}(Y_d; \hspace{0.02cm} \R^d)$ such that $A^{1/2}\Phi$ is a divergence free field in $\R^d$.   Recall that, since $u_n\in H^1_\lambda(Y_d)$, we have that $\grad u_n(y) = \grad v_n(y) + \lambda$, for some $v_n\in H^1_{\rm per} (Y_d)$. This implies that  %By the periodic divergence-curl lemma (see {\it e.g.} \cite[Remark 1.2]{JKO12}), we have
    \begin{linenomath}
    \begin{align}
    \int_{Y_d} A^{1/2}(y)\grad u_n(y)\cdot\Phi(y)dy &= \int_{Y_d}\grad u_n(y)\cdot A^{1/2}(y)\Phi(y)dy\notag \\
    & =\lambda\cdot \int_{Y_d} A^{1/2}(y)\Phi(y)dy + \int_{Y_d}\grad v_n(y)\cdot A^{1/2}(y)\Phi(y)dy\notag \\  
    &  =\lambda\cdot \int_{Y_d} A^{1/2}(y)\Phi(y)dy, \label{perdiv-curllemma}
    \end{align}
    \end{linenomath}
where  the last equality is obtained by integrating by parts the second integral combined with the fact that $A^{1/2}\Phi$ is a divergence free field in $\R^d$. 
In view of convergence \eqref{Astrongconver}, the integral on the left-hand side of \eqref{perdiv-curllemma} converges to $0$. Hence, passing to the limit as $n\to\infty$ in \eqref{perdiv-curllemma} yields 
     \begin{linenomath}
     \begin{equation*}
     \label{lambadaperp}
     0=\lambda\cdot \l(\int_{Y_d} A^{1/2}(y)\Phi(y)dy\r),
     \end{equation*}
     \end{linenomath}
for any $\Phi\in L^2_{\rm per}(Y_d; \hspace{0.02cm} \R^d)$ such that $A^{1/2}\Phi$ is a divergence free field in $\R^d$. Therefore $\lambda\in V^\perp$ which implies that
      \begin{linenomath}
      \begin{equation*}
      \label{kersubsetV}
      \ker(A^\ast) \subseteq V^\perp.
      \end{equation*}
      \end{linenomath}
%In view of assumption (H2), we get that $V^\perp = (\R^d)^\perp =\{0\}$, so that from \eqref{kersubsetV} it follows that
%     \begin{linenomath}
%     \begin{equation*}
%      \ker(A^\ast)=\{0\},
%     \end{equation*}
%     \end{linenomath}
%which implies the positive definiteness of  $A^\ast$.\\
\par Conversely, by \eqref{AdeltatendA} we already know that 
   \begin{linenomath} 
   \begin{equation*}
   \lim_{\delta\to 0} A^\ast_\delta=A^\ast,
   \end{equation*}
   \end{linenomath} 
where $A^\ast_\delta$ is the homogenized matrix associated with $A_\delta = A+\delta I_d$. Since $A_\delta$ is strongly elliptic, the homogenized matrix $A^\ast_\delta$ is given by
    \begin{linenomath}
   \begin{equation}
   \label{minprob}
   A^\ast_\delta \lambda\cdot\lambda = \min\l\{ \int_{Y_d} A_\delta(y)\grad u_\delta(y)\cdot \grad u_\delta(y) dy \hspace{0.02cm}:\hspace{0.02cm} u_\delta\in H^1_\lambda (Y_d)  \r\}.
   \end{equation}
   \end{linenomath}  
Let $\overline{u}_\delta$ be the minimizer  of problem \eqref{minprob}. Therefore, there exists a constant $C>0$ such that 
    \begin{linenomath} 
    \begin{equation*}
    A^\ast_\delta\lambda\cdot\lambda = \int_{Y_d} A_\delta(y)\grad \overline{u}_\delta(y)\cdot \grad \overline{u}_\delta(y) dy= \int_{Y_d}|A^{1/2}_\delta(y)\grad \overline{u}_\delta(y)|^2dy\leq C,
    \end{equation*}
    \end{linenomath} 
which implies that the sequence $\Phi_\delta(y):= A^{1/2}_\delta(y)\grad \overline{u}_\delta(y)$ is bounded in $L^2_{\rm per}(Y_d; \R^d)$. Then,  up to extract a subsequence, we can assume that $\Phi_\delta$ converges weakly to some $\Phi$ in $L^2_{\rm per} (Y_d; \R^d)$.
\par Now, we show that  $A^{1/2}_\delta$ converges strongly to $A^{1/2}$ in $L^\infty_{\rm per}(Y_d)^{d\times d}$.  Since $A(y)$ is a symmetric matrix, there exists an orthogonal matrix-valued function $R$ in $L^\infty_{\rm per}(Y_d)^{d\times d}$ such that
   \begin{equation}
   \notag
   A(y)=R(y)D(y)R^T(y) \qquad\mbox{for a.e. }y\in Y_d,  
   \end{equation}
where $D$ is a diagonal non-negative matrix-valued function in $L^\infty_{\rm per}(Y_d)^{d\times d}$  and $R^T$ denotes the transpose of $R$.  It follows that  $A_\delta (y) = A(y)+\delta I_d = R(y)(D(y)+\delta I_d)R^T(y)$, for a.e. $y\in Y_d$. Hence,
   \begin{equation}
    A^{1/2}_\delta (y) = R(y)(D(y)+\delta I_d)^{1/2}R^T(y)\qquad \mbox{for a.e. }y\in Y_d,\notag
   \end{equation}
which implies that $A^{1/2}_\delta$ converges strongly to $A^{1/2} = RD^{1/2}R^T$  in $L^\infty_{\rm per} (Y_d)^{d\times d}$.

\begin{comment}
\par Let us now show that  $A^{1/2}_\delta$ converges strongly to $A^{1/2}$ in $L^\infty_{\rm per}(Y_d)^{d\times d}$. Since $A_\delta(y)$ and $A(y)$ commute, we deduce that
    \begin{linenomath} 
    \begin{equation*}
    (A_\delta^{1/2}(y)-A^{1/2}(y))(A_\delta^{1/2}(y)+A^{1/2}(y))= A_\delta(y)-A(y)\qquad \mbox{a.e.}\hspace{0.2cm} y\in Y_d.
    \end{equation*}
    \end{linenomath} 
This combined with the positive definiteness of  $A_\delta^{1/2}+A^{1/2}$ implies that,  for a.e. $y\in Y_d$,
     \begin{linenomath} 
     \begin{equation}
     \label{Adelta-A}
     A_\delta^{1/2}(y)-A^{1/2}(y)= (A_\delta(y)-A(y)) (A_\delta^{1/2}(y)+A^{1/2}(y))^{-1} = \delta(A_\delta^{1/2}(y)+A^{1/2}(y))^{-1}. 
     \end{equation}
     \end{linenomath} 
Moreover, we have  
    \begin{linenomath} 
    \begin{equation*}
    A_\delta^{1/2}(y)+A^{1/2}(y)\geq \delta^{1/2}I_d \qquad \mbox{a.e.}\hspace{0.2cm} y\in Y_d,
    \end{equation*}
    \end{linenomath} 
which implies that 
     \begin{linenomath} 
     \begin{equation*}
     \label{upperbound}
    (A_\delta^{1/2}(y)+A^{1/2}(y))^{-1}\leq \delta^{-1/2}I_d\qquad \mbox{a.e.}\hspace{0.2cm} y\in Y_d. 
    \end{equation*}  
    \end{linenomath} 
This combined with \eqref{Adelta-A} yields 
    \begin{linenomath} 
    \begin{equation*}
     0\leq A^{1/2}_\delta (y)-A^{1/2}(y)\leq \delta^{1/2}I_d\qquad \mbox{a.e.}\hspace{0.2cm} y\in Y_d. 
     \end{equation*} 
     \end{linenomath} 
which implies that $A^{1/2}_\delta$ converges strongly to $A^{1/2}$  in $L^\infty_{\rm per} (Y_d)^{d\times d}$. 
\end{comment}

\par Now, passing to the limit as $\delta\to 0$ in 
    \begin{linenomath} 
    \begin{equation*}
    \div(A_\delta^{1/2}\Phi_\delta) = \div (A_\delta\grad \overline{u}_\delta) =0 \qquad\mbox{in } \mathscr{D'}(\R^d),
    \end{equation*}
    \end{linenomath} 
we have
    \begin{linenomath} 
    \begin{equation*}
    \div (A^{1/2}\Phi)=0 \qquad\mbox{in } \mathscr{D'}(\R^d).
    \end{equation*}
    \end{linenomath} 
This along with $\Phi\in L^2_{\rm per}(Y_d; \R^d)$ implies that $\Phi$ is a test function for the set $V$ given by \eqref{setN}.  From \eqref{minprob} it follows that
    \begin{linenomath} 
    \begin{equation*}
    A^\ast_\delta\lambda=\int_{Y_d} A_\delta(y)\grad \overline{u}_\delta(y)dy = \int_{Y_d} A^{1/2}_\delta(y) \Phi_\delta(y)dy.
    \end{equation*}
    \end{linenomath} 
Hence, taking into account the strong convergence of $A^{1/2}_\delta$ in $L^\infty_{\rm per}(Y_d)^{d\times d}$ and the weak convergence of $\Phi_\delta$ in $L^2_{\rm per}(Y_d; \R^d)$, we have 
    \begin{linenomath} 
    \begin{equation*}
    A^\ast\lambda =\lim_{\delta\to 0} A^\ast_\delta \lambda =\lim_{\delta\to 0}
   \int_{Y_d} A^{1/2}_\delta(y) \Phi_\delta(y)dy=\int_{Y_d} A^{1/2}(y)\Phi(y)dy,
    \end{equation*}
    \end{linenomath} 
which implies that $A^\ast\lambda\in V$ since $\Phi$ is a suitable test function for the set $V$. Therefore, for $\lambda\in V^\perp$,
   \begin{linenomath} 
    \begin{equation*}
    A^\ast\lambda\cdot\lambda=0,
    \end{equation*}
    \end{linenomath} 
so that, since $A^\ast$ is a non-negative matrix, we deduce that $\lambda\in\ker(A^\ast)$. In other words, 
    \begin{linenomath} 
    \begin{equation*}
    V^\perp\subseteq \ker(A^\ast),
    \end{equation*}
    \end{linenomath} 
which concludes the proof.
\end{proof}

\section{Two-dimensional and three-dimensional examples}
In this section we provide a geometric setting for which assumptions (H1) and (H2) are fulfilled. We focus on a $1$-periodic rank-one laminates of direction $e_1$ with two phases in $\R^d$, $d=2,3$. Specifically, we assume the existence of two anisotropic phases  $Z_1$ and $Z_2$ of $Y_d$ given by
      \begin{linenomath}
      \begin{equation*}
      Z_1 = (0,\theta)\times (0,1)^{d-1} \quad\mbox{and}\quad Z_2 = (\theta,1)\times (0,1)^{d-1},
      \end{equation*} 
      \end{linenomath}
where $\theta$ denotes the volume fraction of the phase  $Z_1$.
Let $Z_1^\#$ and $Z_2^\#$ be the associated subsets of $\rd$, \textit{i.e.} the open periodic sets
    \begin{linenomath}
    \begin{equation*}
    Z_i^\# := \text{Int}\l(\bigcup_{k\in\mathbb{Z}^d} \l(\overline{Z_i}+k\r)\r) \qquad\mbox{for} \hspace{0.1cm} i =1,2.
    \end{equation*}
    \end{linenomath}
Let $X_1$ and $X_2$ be unbounded connected components of $Z_1^\#$ and $Z_2^\#$ in $\rd$ given by
    \begin{linenomath}
    \begin{equation*}
    X_1 := (0,\theta)\times \R^{d-1} \quad\mbox{and}\quad 
    X_2 := (\theta,1)\times \R^{d-1},
    \end{equation*}
    \end{linenomath}
and we denote by $\partial Z$ the interface $\{y_1=0\}$.   
\par The anisotropic phases are described by two constant, symmetric and non-negative matrices $A_1$ and $A_2$ of $\R^{d\times d}$ which are possibly not positive definite. Hence,  the conductivity matrix-valued function $A\in L^\infty_{\text{per}}(Y_d)^{d\times d}$, given by
      \begin{equation}
      \label{matrixlam}
      A(y_1) := \chi(y_1)A_1 + (1-\chi(y_1))A_2 \qquad\mbox{for}\hspace{0.1cm} y_1\in\R,
      \end{equation}
where $\chi$ is the $1$-periodic characteristic function of the phase $Z_1$, is  not strongly elliptic, \textit{i.e.} \eqref{notstronell} is satisfied.

\subsection{The two-dimensional case with one degenerate phase}
We are interested in two-phase mixtures in $\rtwo$ with one degenerate phase. We specialize to the case where the non-negative and symmetric matrices $A_1$ and $A_2$ of $\R^{2\times 2}$ are such that
     \begin{equation}
     \label{phasesdimtwo}
    A_1=\xi\otimes\xi\qquad\mbox{and}\quad A_2 \hspace{0.1cm}\mbox{\rm is positive definite},
     \end{equation}
for some $\xi\in\rtwo$. The next proposition establishes the algebraic conditions which provide assumptions (H1) and (H2) of Theorem \ref{genthm}.
    \begin{prop}
    \label{propd2}
    Let $A_1$ and $A_2$ be the matrices defined by \eqref{phasesdimtwo}. 
    Assume that $\xi\cdot e_1\neq 0$ and the vectors $ \xi$ and $A_2e_1$ are linearly independent in $\rtwo$.
    Then,  assumptions {\rm (H1)} and {\rm (H2)} are satisfied. In particular,  the homogenized matrix $A^\ast$, given by \eqref{hommat},  associated to the matrix $A$ defined by \eqref{matrixlam} and \eqref{phasesdimtwo} is positive definite.
    \end{prop}
 From Theorem \ref{genthm}, we easily deduce that the energy $\fe$ defined by \eqref{funct} with $A$ given by \eqref{matrixlam} and \eqref{phasesdimtwo}  $\Gamma$-converges to the functional $\fhom$ given by \eqref{homfunc} with conductivity matrix $A^\ast$ defined by \eqref{hommat}. In the present case, the homogenized matrix $A^\ast$ has an explicit expression given in Proposition \ref{prop:Appendix} in the Appendix. % and it is positive definite matrix. 
\begin{proof} 
Firstly, let us prove  assumption (H1). 
We adapt the proof of Step 1 of \cite[Theorem 3.3]{BF19} to  two-dimensional  laminates. In our context, the algebra involved is different due to the scalar setting.
\par Denote  by $u^i_0$  the restriction of the two-scale limit $u_0$ in phase $Z_i$ or $Z_i^\#$ for $i=1,2$.   In view of \eqref{twoscale0}, for any $\Phi(x,y)\in\Cinfcpt(\o\times\rtwo;\hspace{0.09cm} \rtwo)$ with compact support in $\o\times Z^\#_1$, or due to periodicity in $\o\times X_1$,  we 
deduce that
      \begin{linenomath}
      \begin{align}
      0&= -\lim_{\varepsilon\to 0} \varepsilon\int_{\o} A\l(\xe\r)\grad\ue\cdot \Phi\l(x,\xe\r) dx \notag\\
      &=\lim_{\varepsilon\to 0}\int_{\o} \ue  \div_y(A_1\Phi(x,y))\l(x,\xe\r) dx \notag\\ 
      %&=\lim_{\varepsilon\to 0}\left(\varepsilon\int_{\o} \ue(x)  \div_x(A(y)\Phi(x,y))(x,\xe) dx+\int_{\o} \ue(x) \div_y(A(y)\Phi(x,y))(x,\xe) dx\right) \\
      &= \int_{\o\times Z^\#_1} u^1_0(x,y)  \div_y(A_1\Phi(x,y)) dxdy \notag\\%\label{importequ}\\
      &=-\int_{\o\times Z^\#_1} A_1\grad_y u^1_0(x,y)  \cdot\Phi(x,y) dxdy,\notag
      \end{align}
      \end{linenomath}
so that
   \begin{equation}
   \label{fundest1}
   A_1\grad_y u^1_0(x,y)  \equiv 0 \qquad\text{in}\hspace{0.2cm} \Omega\times Z_1^\#.
   \end{equation}
Similarly, taking  $\Phi(x,y)\in\Cinfcpt(\o\times\rtwo; \hspace{0.09cm}\rtwo)$ with compact support in $\o\times Z_2^\#$, or equivalently in $\o\times X_2$, as test function and repeating the same arguments, we obtain
     \begin{equation}
     \label{fundest2}
     A_2\grad_y u^2_0(x,y)  \equiv 0 \qquad\text{in}\hspace{0.2cm} \Omega\times Z_2^\#.
     \end{equation}
Due to \eqref{fundest1}, in phase $Z_1^\#$ we have
         \begin{linenomath}
         \begin{equation*}
         \label{kerAdimtwo}
         \grad_yu^1_0\in\ker (A_1)=\text{Span}(\xi^\perp),
         \end{equation*}
         \end{linenomath}
where $\xi^\perp=(-\xi_2, \xi_1)\in\rtwo$ is perpendicular to $\xi=(\xi_1, \xi_2)$. Hence, $u^1_0$  reads as 
    \begin{equation}
    \label{f7}
   u_0^1(x,y) = \theta^1 (x,\hspace{0.06cm} \xi^\perp\cdot y)\qquad \mbox{a.e.}\hspace{0.1cm}(x,y)\in \o\times X_1,
    \end{equation}
for some function $\theta^1\in L^2(\o\times \R)$. 
On the other hand, since the matrix $A_2$ is positive definite, in phase $Z_2^\#$ the relation \eqref{fundest2} implies that 
      \begin{equation}
      \label{u2noy}
      u^2_0(x,y) = \theta^2(x) \qquad\mbox{a.e.}\hspace{0.1cm}(x,y)\in\o \times X_2,
      \end{equation}
for some function $\theta^2\in L^2(\o)$.  Now, consider a constant vector-valued function $\Phi$ defined on $Y_2$ such that %$\Phi\in C^1_{\text{per}}(Y_2;\hspace{0.09cm}\R^2)$ with
    \begin{equation}
    \label{trasmcond}
    (A_1-A_2)\Phi\cdot e_1 = 0\qquad \text{on}\hspace{0.2cm} \partial Z_1^\#.
    \end{equation}
Note that condition \eqref{trasmcond} is necessary for $\div_y(A(y)\Phi)$ to be an admissible test function for two-scale convergence. In view of \eqref{twoscale0} and \eqref{u2noy},  for any $\varphi\in\Cinfcpt(\o; \Cinfper(Y_2))$, we obtain
    \begin{linenomath}
    \begin{align}
    0&= -\lim_{\varepsilon\to 0} \varepsilon\int_{\o}  A(y)\grad\ue\cdot\Phi\varphi\l(x,\xe\r)dx \notag\\
     &=\lim_{\varepsilon\to 0}  \int_{\o}\ue\div_y(A(y)\Phi\varphi(x,y))\l(x,\xe\r)dx\notag\\
     %&= \int_{\o\times Y_2} u_0(x,y)\div_y(A(y)\Phi(y)\varphi(x,y))dxdy\notag\\
     &= \int_{\o\times Z_1}u^1_0(x,y)\div_y(A_1\Phi\varphi(x,y))dxdy \notag\\
     &\quad + \int_{\o\times Z_2}\theta^2(x)\div_y(A_2\Phi\varphi(x,y))dxdy. \notag%\label{traccia}
     \end{align}
     \end{linenomath}
Take now $\varphi\in\Cinfcpt(\o\times\rtwo)$ and use the periodized function
      \begin{linenomath}
      \begin{equation*}
      \label{perfun}
      \varphi^\#(x,y) :=\sum_{k\in\mathbb{Z}^2}\varphi(x,y+k)
      \end{equation*}
      \end{linenomath}
as new test function. Then, we obtain
  \begin{linenomath} 
  \begin{align}
  0 %&=\int_{\o\times Y_2} u_0(x,y)\div_y(A(y)\Phi(y)\varphi^\#(x,y))dxdy\notag\\
  &=\int_{\o\times Z_1}u^1_0(x,y)\div_y(A_1\Phi\varphi^\#(x,y))dxdy+ \int_{\o\times Z_2}\theta^2(x)\div_y(A_2\Phi\varphi^\#(x,y))dxdy\notag\\
  &=\sum_{k\in\mathbb{Z}^2}\int_{\o\times (Z_1+k)}u^1_0(x,y)\div_y(A_1\Phi\varphi(x,y))dxdy\notag\\
  &\quad +\sum_{k\in\mathbb{Z}^2} \int_{\o\times (Z_2+k)}\theta^2(x)\div_y(A_2\Phi\varphi(x,y))dxdy\notag\\
  & = \int_{\o\times Z^\#_1}u^1_0(x,y)\div_y(A_1\Phi\varphi(x,y))dxdy+ \int_{\o\times Z^\#_2}\theta^2(x)\div_y(A_2\Phi\varphi(x,y))dxdy.  \label{tracciaper}
  \end{align}
  \end{linenomath}
Recall that $A_1=\xi\otimes \xi$, where $\xi$ is such that $\xi\cdot e_1\neq0$. This combined with the linear independence of the vectors
$\xi$ and $A_2e_1$ implies that the linear map 
   \begin{linenomath}
   \begin{equation*}
   \Phi\in\rtwo\mapsto (A_1e_1\cdot\Phi,\hspace{0.03cm}A_2e_1\cdot\Phi)\in\rtwo
   \end{equation*}
   \end{linenomath}
is one-to-one. Hence, for any $f\in\R$, there exists a unique $\Phi\in\rtwo$ such that 
     \begin{equation}
     \label{equality}
     A_1\Phi\cdot e_1= A_2\Phi\cdot e_1 =f.
     \end{equation}
%Hence, the vector-valued function $\Phi$ can be defined as $C^1$ function on $\partial Z_1^\#$ such that $\Phi(y)$ satisfies \eqref{equality}, which implies, in particular, the transmission condition \eqref{trasmcond}. Then, by \textit{e.g.}  Whitney's extension theorem, $\Phi$ can be extended as a function in $C^1_{\text{per}}(Y_2)$.
In view of the arbitrariness of $f$ in \eqref{equality}, we can choose  $\Phi$  such that 
     \begin{equation}
     \label{condPhi}
     A_1e_1\cdot \Phi= A_2e_1\cdot \Phi = 1\qquad \mbox{on}\hspace{0.2cm}\partial Z^\#_1.
     \end{equation}     
Since $A_1 \grad_y u^1_0 =0$ in the distributional sense and  $A_1=\xi\otimes\xi$, we deduce that $u^1_0$ is constant along the direction $\xi$. Using Fubini's theorem, we may integrate along straight lines parallel to the vector $\xi$ where integration by parts is allowed. Therefore, performing an integration by parts in \eqref{tracciaper} combined with \eqref{condPhi}, it follows that for any $\varphi\in C^\infty_{\rm c}(\Omega\times\mathbb{R}^2)$,
    \begin{linenomath}
    \begin{align}
    0 %&= \int_{\o\times Z_1^\#} u^1_0(x,y)\div_y (A_1\Phi(y)\varphi(x,y))dxdy + \int_{\o\times Z_2^\#} \theta^2(x)\div_y (A_2\Phi(y)\varphi(x,y))dxdy\notag\\
    %&=\int_{\o\times \partial Z} u^1_0(x,y) A_1\Phi\cdot e_1\varphi(x,y)dx\dboundary - \int_{\o\times \partial Z} \theta^2(x) A_2\Phi\cdot e_1\varphi(x,y)dx\dboundary\notag\\
    &=\int_{\o\times \partial Z} v_0(x,y)\varphi(x,y)dx\dboundary, \notag%\label{tracciav}
    \end{align}
    \end{linenomath}
where we have set $v_0(x,y):= u^1_0(x,y) - \theta^2(x)$.  We conclude that $v_0(x,\cdot)$ has a trace on $\partial Z$ for a.e. $x\in\Omega$ satisfying 
  \begin{linenomath}
 \begin{equation}
       \label{traceequ}
       v_0(x,\cdot) = 0 \qquad\text{on}\hspace{0.2cm}\partial Z.
       \end{equation}
      \end{linenomath}

\begin{comment}
       Note that 
    \begin{linenomath}
    \begin{align*}
    0&=\int_{\o\times\R^2} \theta^2(x)\div_y (A(y)\Phi(y)\varphi(x,y))dxdy\\
    &= \int_{\o\times Z_1^\#} \theta^2(x)\div_y (A_1\Phi(y)\varphi(x,y))dxdy+ \int_{\o\times Z_2^\#} \theta^2(x)\div_y (A_2\Phi(y)\varphi(x,y))dxdy.
    \end{align*}
    \end{linenomath}
This along  with \eqref{tracciaper} implies that 
    \begin{linenomath}
    \begin{align}
    0 %&= \int_{\o\times Z_1^\#} u^1_0(x,y)\div_y (A_1\Phi(y)\varphi(x,y))dxdy + \int_{\o\times Z_2^\#} \theta^2(x)\div_y (A_2\Phi(y)\varphi(x,y))dxdy\notag\\
    &=\int_{\o\times Z_1^\#} u^1_0(x,y)\div_y (A_1\Phi(y)\varphi(x,y))dxdy - \int_{\o\times Z_1^\#} \theta^2(x)\div_y (A_1\Phi(y)\varphi(x,y))dxdy\notag\\
    &=\int_{\o\times Z_1^\#} v_0(x,y)\div_y(A_1\Phi(y)\varphi(x,y))dxdy.\label{tracciav}
    \end{align}
    \end{linenomath}
Since $A_1\grad_y v_0=0$ in $\o\times Z_1^\#$ in the distributional sense and due to \eqref{condPhi} and the arbitrariness of $\varphi$, an integration by parts of \eqref{tracciav} implies that  $v_0(x,\cdot)$ has a trace on $\partial Z$ for a.e. $x\in\o$ satisfying 
       \begin{equation}
       \label{traceequ}
       v_0(x,\cdot) = 0 \qquad\text{on}\hspace{0.2cm}\partial Z.
       \end{equation}
\end{comment}

Recall that  $\partial Z= \{y_1=0\}$. Fix $x\in\o$. Taking into account \eqref{f7} and \eqref{u2noy}, the equality \eqref{traceequ} reads as 
     \begin{linenomath}
     \begin{equation*}
     \theta^1(x, \hspace{0.03cm}\xi_1y_2)= \theta^2(x) \qquad\text{on}\quad \partial Z.
     \end{equation*}
     \end{linenomath}
Since $\xi\cdot e_1 \neq 0$, it follows that $\theta^1$ only depends on $x$ so that $u^1_0(x,y)$ agrees with $\theta^2(x)$. Finally, we conclude that $u_0(x,y) := \chi(y_1)u_0^1(x,y) + (1-\chi(y_1))u^2_0(x,y) $ is independent of $y$ and hence (H1) is satisfied. 

%the map $y_2\in\mathbb{R}\mapsto z:=\xi_1y_2\in\mathbb{R}$ is actually a change of variable. Hence 
%        \begin{equation}
%        \label{const}
%        \theta^1(x,z) = \theta^2 (x) \qquad \text{a.e.}\quad z\in\mathbb{R}.
%        \end{equation}
%Let $z_0\in\R$ be a Lebesgue point for $\theta^1$. From \eqref{const},  we infer that $\theta^1 (x, \cdot)$ is a constant function for a.e. $z\in\R$, hence, since $z_0$ is a Lebesgue point, $\theta^1(x, z_0)$ is a constant. 

\par
Let us prove assumption (H2). 
The proof is a variant of the Step 2 of   \cite[Theorem 3.4]{BF19}. For arbitrary $\alpha, \beta\in\R$, let $\Phi$ be a vector-valued function given by 
     \begin{equation}
     \label{vectNdtwo}
     A^{1/2}(y)\Phi(y) := \chi(y_1)\alpha\xi + (1-\chi(y_1))(\alpha\xi+\beta e_2) \qquad\mbox{for a.e.}\hspace{0.1cm} y\in\rtwo.
     \end{equation}
Such a vector field $\Phi$ does exist, since $\xi$ is in the range of $A_1$ and thus the right-hand side of \eqref{vectNdtwo} belongs pointwise to the range of $A$, or equivalently to the range of $A^{1/2}$. Moreover, the difference of two constant phases in \eqref{vectNdtwo} is orthogonal to the laminate direction $e_1$, so that $A^{1/2}\Phi$ is a laminate divergence free periodic field in $\rtwo$. Its average value is given by
      \begin{linenomath} 
      \begin{equation*}
      N:=\int_{Y_2} A^{1/2}(y)\Phi(y)dy = \alpha\xi+(1-\theta)\beta e_2.
      \end{equation*}
      \end{linenomath}
Hence, due to $\xi\cdot e_1\neq 0$ and the arbitrariness of $\alpha, \beta$, the set of the vectors $N$ spans $\rtwo$, which yields assumption (H2). \par  From Proposition \ref{prop:positivedefinite}, it immediately follows that the homogenized matrix $A^\ast$ is positive definite.
For the reader's convenience, the proof of explicit formula of $A^\ast$ is postponed to Proposition \ref{prop:Appendix} in the Appendix.
\end{proof}

\subsection{The three-dimensional case with both degenerate phases}
We are going to deal with three-dimensional laminates where both phases are degenerate. We assume that the symmetric  and non-negative matrices $A_1$ and $A_2$ of $\R^{3\times 3}$ have rank two, hence, there exist $\eta_1, \eta_2\in\R^3$ such that 
       \begin{equation}
       \label{kernranktwo}
       \ker (A_i) =  {\rm Span}(\eta_i) \qquad\mbox{for}\hspace{0.1cm} i=1,2. 
       \end{equation}
The following proposition gives the algebraic conditions so that assumptions required by Theorem \ref{genthm} are satisfied.
    \begin{prop}
    \label{thmcasematrixdege} 
    Let $\eta_1$ and $\eta_2$ be the vectors in $\R^3$ defined by \eqref{kernranktwo}.
   Assume that the vectors $\{e_1, \eta_1, \eta_2\}$ as well as $\{A_1e_1, A_2e_1\}$ are linearly independent in $\R^3$. Then, assumptions {\rm (H1)} and {\rm (H2)} are satisfied. In particular, the homogenized matrix $A^\ast$  given by \eqref{hommat} and  associated to the conductivity matrix $A$ given by \eqref{matrixlam} and \eqref{kernranktwo} is positive definite. 
   \end{prop} 
Invoking again Theorem  \ref{genthm}, the energy $\fe$ defined by \eqref{funct} with $A$ given by \eqref{matrixlam} and \eqref{kernranktwo}, $\Gamma$-converges for the weak topology  of $L^2(\o)$ to $\fhom$  where the effective conductivity $A^\ast$ is given by \eqref{hommat}. As in two-dimensional laminate materials, $A^\ast$ has an explicit expression  (see Proposition \ref{prop:Appendix} in the Appendix). %and it is positive definite  
\begin{proof} 
Let us first check assumption (H1). The proof is an adaptation  of the first step of \cite[Theorem 3.3]{BF19}. Same arguments as in the proof of Proposition \ref{propd2} show that
    \begin{equation}
    \label{fundestdimthree}
    A_i\grad_y u^i_0(x,y) \equiv 0 \qquad \mbox{in}\hspace{0.1cm} \o\times Z_i^\#\quad\mbox{for}\hspace{0.1cm} i=1,2.
    \end{equation}
In view of \eqref{kernranktwo} and \eqref{fundestdimthree}, in phase $Z_i^\#$, $u^i_0$ reads as 
      \begin{equation}
      \label{functionui0}
      u^i_0(x,y) = \theta^i(x, \eta_i\cdot y) \qquad \mbox{a.e.}\hspace{0.1cm} (x,y)\in\o\times X_i, 
      \end{equation}
for some function $\theta^i\in L^2(\o\times\R) $ and $i=1,2$. 
Now, consider a constant vector-valued function $\Phi$ on $Y_3$ such that the transmission condition \eqref{trasmcond} holds. In view of \eqref{twoscale0}, for any $\varphi\in\Cinfcpt(\o,\hspace{0.08cm} \Cinfper(Y_3))$, we  obtain
     \begin{linenomath}
     \begin{align}
     0&=-\lim_{\varepsilon\to 0} \varepsilon\int_{\o} A(y)\grad\ue\cdot\Phi\varphi\l(x,\xe\r)dx\notag\\
     &= \int_{\o\times Z_1}u^1_0(x,y)\div_y(A_1\Phi\varphi(x,y))dxdy\notag \\
     &\quad + \int_{\o\times Z_2}u^2_0(x,y)\div_y(A_2\Phi\varphi(x,y))dxdy. \label{d3}
     \end{align}
     \end{linenomath}
Take $\varphi\in\Cinfcpt(\o\times\R^3)$. Putting the periodized function
     \begin{linenomath}
     \begin{equation*}
     \varphi^\# (x,y) := \sum_{k\in\mathbb{Z}^3} \varphi(x, y+k)
     \end{equation*}
     \end{linenomath}
as test function in \eqref{d3}, we get
    \begin{equation}
    \label{dimthree}
    \int_{\o\times Z_1^\#} u^1_0(x,y)\div_y(A_1\Phi\varphi(x,y)) dxdy +\int_{\o\times Z_2^\#} u^2_0(x,y)\div_y(A_2\Phi\varphi(x,y)) dxdy=0.
    \end{equation}
Since the vectors $A_1e_1$ and $A_2e_1$ are independent in $\R^3$, the linear map
     \begin{linenomath}
     \begin{equation*}
     \Phi\in\mathbb{R}^3\mapsto(A_1 e_1\cdot\Phi,\hspace{0.06cm} A_2e_1\cdot\Phi)\in\rtwo
     \end{equation*}
     \end{linenomath}
is surjective. In particular, for any $f\in\R$, there exists $\Phi\in\R^3$ such that
      \begin{equation}
      \label{tr1}
      A_1\Phi\cdot e_1=A_2\Phi\cdot e_1=f.
      \end{equation} 
%Hence, one can define the vector-valued function $\Phi$ as $C^1$ function on $\partial Z_1^\#$ such that $\Phi(y)$ satisfies \eqref{tr1} which implies, in particular, the transmission condition \eqref{trasmcond}. Then, $\Phi$ can be extended as a function in $C^1_{\text{per}}(Y_3)$ by  Whitney's extension Theorem. 
In view of the arbitrariness of $f$ in \eqref{tr1}, we can choose $\Phi$ such that \eqref{condPhi} is satisfied.  
Due to \eqref{fundestdimthree} and  \eqref{kernranktwo}, we deduce that $u^i_0$ is constant along the plane $\Pi_i$ perpendicular to $\eta_i$, for $i=1,2$. This implies that, thanks to Fubini's theorem, we may integrate along the plane $\Pi_i$ where an integration by part may be performed. Hence, an integration by parts  in \eqref{dimthree}  combined  with
\eqref{condPhi},  yields for any $\varphi\in\Cinfcpt(\o\times \R^3)$,
  \begin{linenomath}
  \begin{equation*}
  \int_{\o\times\partial Z}\l[u^1_0(x,y)-u^2_0(x,y)  \r]\varphi(x,y)dx\dboundary =0,
  \end{equation*}
  \end{linenomath}
which implies that 
     \begin{equation}
     \label{traceeq1}
     u^1_0(x,\cdot) = u^2_0(x,\cdot) \qquad \text{on}\quad \partial Z.
     \end{equation}

%after an integration by parts with respect to $y$ of the first integral in \eqref{dimthree}, we get
%  \begin{linenomath}
%   \begin{align*}
%   \l|\int_{\o\times\partial Z^\#_1} u^1_0(x,y)\varphi(x,y)dx\dboundary  \r| & \leq %\int_{\o\times Z^\#_2} |u^2_0(x,y) \varphi(x,y) \div_y(A_2\Phi(y))|dxdy   \\
%        \int_{\o\times Z^\#_2} |u^2_0(x,y) \grad_y\varphi(x,y)\cdot A_2\Phi(y)|dxdy  \\
%   &\leq C\|\grad_y\varphi\|_{L^2(\o\times \R^3)} . 
%   \end{align*}
%   \end{linenomath}
%It follows that $u^1_0(x,\cdot)$ has a trace on $\partial Z^\#_1$ for a.e. $x\in\o$. 
 
%Now, consider the interface $\partial Z$ between the connected component $X_1$ and $X_2$ of $Z^\#_1$ and $Z^\#_2$ in $\R^3$.
 
Fix $x\in\o$ and recall that $\partial Z=\{y_1=0\}$. In view of  \eqref{functionui0}, the relation \eqref{traceeq1} reads as
       \begin{linenomath}
       \begin{align}
       \label{te1}
       \theta^1(x,\hspace{0.1cm} b_1y_2 + c_1y_3) = \theta^2(x,\hspace{0.1cm} b_2y_2 + c_2y_3) \qquad \text{on}\quad \partial Z ,
       \end{align}
       \end{linenomath}
with $\eta_i = (a_i, b_i, c_i)$ for $i=1,2$.
Due to the independence of $\{e_1, \eta_1, \eta_2\}$ in $\mathbb{R}^3$, the linear map $( y_2, y_3)\in\R^2\mapsto (z_1, z_2)\in\R^2$ defined by
       \begin{linenomath}
       \begin{equation*}
       z_1 :=b_1y_2 + c_1y_3, \qquad z_2 := b_2y_2 + c_2y_3,
       \end{equation*}
       \end{linenomath}
is a change of variables so that \eqref{te1} becomes
     \begin{linenomath}
     \begin{equation*}
     \label{constfunc}
     \theta^1 (x, z_1) = \theta^2(x, z_2) \qquad \text{a.e.} \quad z_1, z_2\in\mathbb{R}.
     \end{equation*} 
     \end{linenomath}
This implies that $\theta^1$ and $\theta^2$  depend only on $x$ and thus $u^1_0$ and $u^2_0$ agree with some function $u\in L^2(\o)$. Finally, we conclude that $u_0(x,y)= \chi(y_1)u^1_0(x,y) + (1-\chi(y_1))u^2_0(x,y)$ is independent of $y$ and hence (H1) is satisfied. 
%Indeed take a Lebesgue point $\overline{z}_1$ for $\theta^1$ and then a Lebesgue point $\overline{z}_2$ for $\theta^2$. Thanks to \eqref{constfunc}, $\theta^1(x, \cdot)$ and $\theta^2(x, \cdot)$ are equal to some constant a.e. $z\in\R$ and hence $\theta^1(x, \overline{z}_1)$ and $\theta^2(x, \overline{z}_2)$ are the same constant. It implies that $u^1_0$ and $u^2_0$ are independent of $y$ and hence (H1)  holds.\\\vspace{0.1cm}
\par It remains to prove  assumption (H2). To this end,
let $E$ be the subset of $\mathbb{R}^3\times\mathbb{R}^3$ defined by
    \begin{equation}
    \label{setE}
    E:= \{ (\xi_1, \xi_2)\in\mathbb{R}^3\times\mathbb{R}^3\hspace{0.02cm}:\hspace{0.02cm} (\xi_1-\xi_2)\cdot e_1 =0,  \hspace{0.2cm}\xi_1\cdot \eta_1=0, \hspace{0.2cm}\xi_2\cdot \eta_2=0   \}.
    \end{equation} 
For $(\xi_1, \xi_2)\in E$, let $\Phi$ be  the vector-valued function defined by
      \begin{equation}
      \label{vectPhidimthree}
      A^{1/2}(y)\Phi(y) := \chi(y_1)\xi_1+(1-\chi(y_1))\xi_2 \qquad \mbox{a.e.}\hspace{0.1cm}y\in\R^3.
      \end{equation} 
The existence of such a vector field $\Phi$ is guaranteed by the conditions $\xi_i\cdot\eta_i=0$, for $i=1,2$, which imply that $\xi_i$ belongs to the range of $A_i$ and hence the right-hand side of \eqref{vectPhidimthree} belongs pointwise to the range of $A$, or equivalently to the range of $A^{1/2}$. Moreover, since the difference of the phases $\xi_1$ and $\xi_2$ is orthogonal to the laminate direction $e_1$, $A^{1/2}\Phi$ is a laminate divergence free periodic field in $\R^3$. Its average value is given by
     \begin{linenomath}
     \begin{equation*}
     N:= \int_{Y_3}A^{1/2}(y)\Phi(y)dy = \theta\xi_1+(1-\theta)\xi_2.
     \end{equation*}
     \end{linenomath}
Note that $E$ is a linear subspace of $\R^3\times\R^3$ whose dimension is three. Indeed, let $f$ be the linear map defined by
    \begin{linenomath}
    \begin{equation*}
    (\xi_1, \xi_2)\in \R^3\times\R^3\mapsto \l((\xi_1-\xi_2)\cdot e_1,\hspace{0.1cm} \xi_1\cdot\eta_1,\hspace{0.1cm} \xi_2\cdot\eta_2 \r)\in\R^3.
    \end{equation*}
    \end{linenomath}     
%Recall that $\R^3\times\R^3$ is isomorphic to $\R^6$ as additive group. Hence, we can
If we identity the pair $(\xi_1,\xi_2)\in\R^3\times \R^3$ with the vector $(x_1, y_1, z_1, x_2, y_2, z_2)\in\R^6$, with $\xi_i= (x_i, y_i, z_i)$, for $i=1,2$, the associated matrix $M_f\in\R^{3\times 6}$ of $f$ is given by
      \begin{linenomath}
      \begin{equation*}
      M_f:=
      \begin{pmatrix}
      1 & 0 & 0 & -1 & 0 & 0\\
      a_1 & b_1 & c_1 & 0 & 0 & 0\\
      0 & 0 & 0 & a_2 & b_2 & c_2
      \end{pmatrix},
      \end{equation*}
      \end{linenomath}
with $\eta_i=(a_i, b_i, c_i)$, $i=1,2$. In view of the linear independence of $\{ e_1, \eta_1, \eta_2 \}$, the rank of $M_f$ is three, which implies that the dimension of kernel $\ker(f)$ is also three. Since the kernel $\ker(f)$ agrees with $E$, we conclude that the dimension of $E$ is three. 
\par Now, let $g$ be the linear map defined by
     \begin{linenomath}
     \begin{equation*}
     (\xi_1, \xi_2)\in E\mapsto \theta\xi_1+(1-\theta)\xi_2\in\mathbb{R}^3.
     \end{equation*}
     \end{linenomath}    
Let us show that $g$ is invertible.
To this end, consider $(\xi_1, \xi_2)\in\ker(g)$. From the definition of the map $g$, $\ker(g)$ consists of  all vectors $(\xi_1, \xi_2)\in E$ of the form
     \begin{linenomath}
     \begin{equation}
     \label{vectkerg}
     \l(\xi_1, \hspace{0.1cm} \frac{\theta}{\theta-1}\xi_1\r).
     \end{equation}
     \end{linenomath} 
In view of the definition of $E$ given by \eqref{setE}, the vector  \eqref{vectkerg}  satisfies the conditions
    \begin{linenomath}
    \begin{equation*}
    \l(\xi_1-\frac{\theta}{\theta-1}\xi_1 \r)\cdot e_1=0, \hspace{0.3cm} \xi_1\cdot\eta_1=0, \hspace{0.3cm} \frac{\theta}{\theta-1}\xi_1\cdot\eta_2=0.
    \end{equation*}
    \end{linenomath}
This combined with the linear independence of $\{e_1, \eta_1, \eta_2 \}$ implies that
       \begin{linenomath}
       \begin{equation*}
       \xi_1\in\{e_1,\eta_1, \eta_2 \}^\perp =\{0\}. 
       \end{equation*}
       \end{linenomath}
Hence, $\ker(g)=\{(0,0)\}$ which implies along with the fact that the dimension of $E$ is three that  $g$ is invertible. This proves that all the vectors of $\R^3$ can be attained through the map $g$ so that assumption (H2) is satisfied.
\par Thanks to Proposition \ref{prop:positivedefinite}, the homogenized matrix $A^\ast$ turns out to be positive definite. The proof of the explicit expression of $A^\ast$ is given in Proposition \ref{prop:Appendix} in the Appendix.
\end{proof}

\section{A two-dimensional counter-example}
In this section we are going to construct a counter-example of two-dimensional laminates with two degenerate  phases, where the lack of assumption (H1) provides an anomalous asymptotic behaviour of the functional $\fe$ \eqref{funct}.
%The following counter-example shows that if the phases of a two-dimensional laminate materials are too degenerate, we may not get enough information about the oscillations of the two-scale limit so that the assumption (H1) may fail.  \\
\par Let $\o:=(0,1)^2$ and let $e_2$ be the laminate direction. We assume that the non-negative and symmetric matrices $A_1$ and $A_2$ of $\R^{2\times 2}$ are given by 
      \begin{linenomath}
      \begin{equation*}
      A_1=e_1\otimes e_1 \quad\text{and}\quad A_2=ce_1\otimes e_1,
      \end{equation*}
      \end{linenomath}
for some positive constant $c>1$. The presence of $c\neq 1$ is  essential to have oscillation in the conductivity matrix $A$. In the present case, the matrix-valued conductivity $A$ is given by 
       \begin{linenomath}
       \begin{equation}
       \label{condmatrcountex}
       A(y_2) := \chi(y_2)A_1+(1-\chi(y_2))A_2= a(y_2)e_1\otimes e_1\qquad \mbox{for}\hspace{0.1cm}y_2\in\R,
       \end{equation}
      \end{linenomath}
with
      \begin{linenomath}
      \begin{equation}
      \label{acountex}
      a(y_2) := \chi(y_2) + c(1-\chi(y_2))\geq 1.
      \end{equation}
      \end{linenomath}
Thus, the energy $\fe$, defined by \eqref{funct}  with $A(y)$ given by \eqref{condmatrcountex} and \eqref{acountex} becomes
        \begin{linenomath} 
        \begin{align}
        \label{c10}
        \fe(u)=
        \begin{dcases}
        \int_{\o}\l[a\l(\xetwo\r)\l(\partial u\over \partial x_1\r)^2 + |u|^2\r] dx,&\mbox{if $u\in  H^1_0((0,1)_{x_1}; L^2(0,1)_{x_2}),$}\\ %H^1_0(\Omega)
         &\\
        \hspace{2cm}\infty,&\mbox{if $u\in L^2(\o)\setminus H^1_0((0,1)_{x_1}; L^2(0,1)_{x_2}) $}.
        \end{dcases}
        \end{align} 
        \end{linenomath}
We denote by $\ast_1$ the convolution with respect to the variable $x_1$, \textit{i.e.} for $f\in L^1(\R^2)$ and $g\in L^2(\R^2)$
                \begin{linenomath}
                \begin{equation*}
                (f\ast_1g)(x_1,x_2) = \int_{\mathbb{R}}f(x_1-t, x_2)g(t,x_2)dt.
                \end{equation*}
                \end{linenomath}
Throughout this section, $c_{\theta}$ denotes the positive constant given by  
         \begin{linenomath}
         \begin{equation}
         \label{ctheta}
         c_\theta:= c\theta+1-\theta,
         \end{equation}
         \end{linenomath}
where $\theta\in (0,1)$ is the volume fraction of the phase $Z_1$ in $Y_2$.
The following result proves the $\Gamma$-convergence of $\fe$ for the weak topology of $L^2(\o)$ and provides two alternative expressions of the $\Gamma$-limit, one of that seems nonlocal due to presence of convolution term (see Remark \ref{rmk:BDrepresentation} below). 
  \begin{prop}
   \label{prop1}
   Let $\fe$ be the functional defined by \eqref{c10}. Then, $\fe$ $\Gamma$-converges for the weak topology of $L^2(\o)$ to the functional defined by 
              \begin{linenomath}
             \begin{equation*}
               \label{c12}
               \mathscr{F} (u) :=
               \begin{dcases}
                 \int_{0}^{1}dx_2\int_{\mathbb{R}} \frac{1}{\hat{k}_0(\lambda_1)}|\fourtra_2 (u)(\lambda_1, x_2)|^2d\lambda_1,&\mbox{if $u\in H^1_0((0,1)_{x_1}; L^2(0,1)_{x_2})$},\\
                  &\\
                \hspace{2cm}\infty,&\mbox{if $u\in L^2(\o)\setminus H^1_0((0,1)_{x_1}; L^2(0,1)_{x_2})$},
                 \end{dcases}
               \end{equation*}
                \end{linenomath}
    where  $\fourtra_2 (u)(\lambda_1, \cdot)$ denotes the Fourier transform on $L^2(\R)$ of parameter $\lambda_1$ with respect to the variable $x_1$ of the function $x_1\mapsto u(x_1,\cdot)$  extended by zero outside $(0,1)$ and 
       \begin{linenomath}
       \begin{equation}
           \label{defk0}
           \hat{k}_0(\lambda_1):= \int_{0}^{1}\frac{1}{4\pi^2a(y_2)\lambda_1^2 +1}dy_2.
           \end{equation} 
       \end{linenomath}
    The $\Gamma$-limit $\mathscr{F}$ can be also expressed as 
              \begin{linenomath}
              \begin{align}
              \label{funccontex}
              \mathscr{F}(u) :=
              \begin{dcases}
              \int_{0}^{1}dx_2\int_{\mathbb{R}}\l\{\frac{c}{c_\theta}\l(\frac{\partial u}{\partial x_1}\r)^2(x_1,x_2) +[\sqrt{\alpha}u(x_1,x_2) + (h\ast_1u)(x_1,x_2)]^2\r\}dx_1,&\\
              \hspace{6.7cm}\mbox{if $u\in H^1_0((0,1)_{x_1}; L^2(0,1)_{x_2})$},&\\
              &\\
              \hspace{2cm}\infty,\hspace{4.2cm}\mbox{if $u\in L^2(\o)\setminus H^1_0((0,1)_{x_1}; L^2(0,1)_{x_2})$},&
              \end{dcases}
              \end{align}
              \end{linenomath}
      where %$u$ is extended by zero with respect to the variable $x_1$ outside $(0,1)$, 
      $c_\theta$ is given by \eqref{ctheta} and  $h$ is a real-valued function in  $L^2(\R)$ defined by means of its Fourier transform $\fourtra_2$ on $L^2(\R)$
            \begin{equation}
            \label{four2}
            \fourtra_2(h)(\lambda_1) :=  \sqrt{\alpha +f(\lambda_1)}-\sqrt{\alpha},
            \end{equation}
     where $\alpha$ and $f$ are given by 
                  \begin{linenomath}
                  \begin{equation}
                  \label{alphaf}
                  \alpha:= \frac{c^2\theta +1-\theta}{c_\theta^2}> 0, \qquad f(\lambda_1):=\frac{(c-1)^2\theta(\theta-1)}{c^2_\theta}\frac{1}{c_\theta4\pi^2\lambda_1^2 + 1}.
                  \end{equation}
                  \end{linenomath}
     Moreover, any two-scale limit $u_0(x,y)$ of a sequence $\ue$ with  bounded energy $\fe$  depends on the variable $y_2\in Y_1$. 
   \end{prop}
   \begin{remark}
    From  \eqref{alphaf}, we can deduce that
        \begin{linenomath}
        \begin{equation*}
        \alpha +f(\lambda_1) = {1\over c^2_\theta(c_\theta4\pi^2\lambda_1^2 + 1)}\l\{(c^2\theta+1-\theta)c_\theta4\pi^2\lambda_1^2+ [(c-1)\theta+1]^2   \r\}> 0 \qquad \forall\lambda_1\in\R,
        \end{equation*}
        \end{linenomath}
   so that the Fourier transform of $h$ is well-defined.
   \end{remark}

%Since the two-scale limit $u_0(x,y)$  depends on the fast variable $y_2\in Y_1$, assumption (H1)  fails so that the $\Gamma$-limit given by Theorem \ref{genthm} \textit{a priori} does not hold.

%\begin{equation}
%             \frac{f(\lambda_1)}{\sqrt{\alpha +f(\lambda_1)}+\sqrt{\alpha}} = O(\lambda_1^{-2}),
%\end{equation}

\begin{proof} We divide the proof into three steps.

\smallskip
\noindent\textit{Step 1} - $\Gamma$-$\liminf$ inequality.\\
Consider a sequence $\{\ue \}_\varepsilon$ converging weakly in $L^2(\o)$ to $u\in L^2(\o)$.  Our aim is to prove  that
      \begin{equation}
      \label{Gammaliminf1}
      \liminf_{\varepsilon\to 0 } \fe (\ue)\geq \mathscr{F}(u).
      \end{equation}
If the lower limit is $\infty$ then \eqref{Gammaliminf1} is trivial. Up to a subsequence, still indexed by $\varepsilon$, we may assume that $\liminf\fe(\ue)$ is a limit and we may assume henceforth that, for some $0<C<\infty$,
         \begin{linenomath}
        \begin{equation}
        \label{liminf2}
        \fe(\ue)\leq C. 
        \end{equation}
        \end{linenomath}
It follows that the sequence  $\ue$ is  bounded in $L^2(\o)$ and according to \cite[Theorem 1.2]{A92}, a subsequence, still indexed by $\varepsilon$, of that sequence  two-scale converges to some $u_0(x,y)\in L^2(\o\times Y_2)$. In other words,
     \begin{equation}
     \label{twoscalecont}
     \ue\leftharpoonupeq u_0.
     \end{equation}
In view of \eqref{acountex}, we know that $a\geq 1$ so that, thanks to \eqref{liminf2}, for another subsequence (not relabeled) we have
    \begin{equation}
    \label{convpartder}
    \frac{\partial \ue}{\partial x_1} \leftharpoonupeq \sigma_0(x,y) \qquad\mbox{with}\hspace{0.2cm}\sigma_0\in L^2(\o\times Y_2).
    \end{equation}
In particular, 
    \begin{linenomath}
    \begin{equation}
    \label{gratwoscaleconv0}
    \varepsilon\frac{\partial \ue}{\partial x_1} \leftharpoonupeq 0.
    \end{equation}
    \end{linenomath}
Take $\varphi\in \Cinfcpt(\o;\hspace{0.1cm}\Cinfper(Y_2))$. By integration by parts, we obtain 
   \begin{linenomath}
   \begin{equation*}
   \varepsilon\int_{\o}\frac{\partial \ue}{\partial x_1}\varphi\l(x,\xe\r)dx = - \int_{\o}\ue\l(\varepsilon\frac{\partial \varphi}{\partial x_1}\l(x,\xe\r)+\frac{\partial \varphi}{\partial y_1}\l(x,\xe\r) \r)dx.
   \end{equation*}
   \end{linenomath}    
Passing to the limit in both terms with the help of \eqref{twoscalecont} and \eqref{gratwoscaleconv0} leads to 
   \begin{linenomath}
   \begin{equation*}
   0= - \int_{\o\times Y_2}u_0(x,y)\frac{\partial \varphi}{\partial y_1}(x,y)dxdy,
   \end{equation*}
   \end{linenomath}
which implies that 
         \begin{linenomath}
         \begin{equation}
         \label{u10indpey1}
         u_0(x,y) \quad\mbox{is independent of}\hspace{0.1cm} y_1.
         \end{equation}
         \end{linenomath}   
Due to the link between two-scale and weak $L^2$-convergences (see  \cite[Proposition 1.6]{A92}), we have 
     \begin{equation}
     \label{c3}
     \ue\rightharpoonup u(x)=\int_{Y_1} u_0(x,y_2)dy_2\qquad\mbox{weakly in $L^2(\o)$.}
     \end{equation}
Now consider  $\varphi\in C^\infty(\overline{\o};\hspace{0.07cm} C^\infty_{\text{per}}(Y_2))$ such that
       \begin{equation}
       \label{divfreeconex}
       \frac{\partial\varphi}{\partial y_1} (x,y)=0.
       \end{equation}
Since $\ue\in H^1_0((0,1)_{x_1}; L^2(0,1)_{x_2})$, an integration by parts leads us to 
      \begin{linenomath}
      \begin{align*}
      \int_{\o} \frac{\partial\ue}{\partial x_1}\varphi\l(x, y \r)dx = -\int_{\o} \ue\frac{\partial\varphi}{\partial x_1}\l(x, y\r)dx.
      \end{align*}
      \end{linenomath}
In view of the convergences \eqref{twoscalecont} and \eqref{convpartder} together with \eqref{u10indpey1}, we can pass to the two-scale limit in the previous expression and we obtain 
      \begin{linenomath}
      \begin{align}
      \label{formgradchi}
      \int_{\o\times Y_2}\sigma_0(x,y)\varphi(x,y)dxdy 
      &=-\int_{\o\times [0,1)} u_0(x,y) \frac{\partial \varphi}{\partial x_1}(x, y )dxdy_2.
      \end{align}
      \end{linenomath}
Varying $\varphi\in\Cinfcpt(\o;\hspace{0.03cm} \Cinfper(Y_2))$, the left-hand side of \eqref{formgradchi} is bounded by a constant times $\|\varphi\|_{L^2(\o\times [0,1))}$ so that the right-hand side is a linear and continuous form in $\varphi\in L^2(\o\times Y_2)$. By the Riesz representation theorem, there exists $g\in L^2(\o\times Y_2)$ such that, for any $\varphi\in \Cinfcpt(\o;\hspace{0.03cm} \Cinfper(Y_2))$, 
     \begin{linenomath}
    \begin{equation*}
    \int_{\o\times Y_2}u_0(x,y_2) \frac{\partial \varphi}{\partial x_1}(x,y)dxdy = \int_{\o\times Y_2} g(x,y)\varphi(x,y) dxdy,
    \end{equation*}
    \end{linenomath}
which yields
    \begin{linenomath}
    \begin{equation}
    \label{partialu0x1}
    \frac{\partial u_0}{\partial x_1}(x,y_2) \in L^2(\o\times Y_1).
    \end{equation}
    \end{linenomath}
Then, an integration by parts with respect to $x_1$ of the right-hand side of \eqref{formgradchi}  yields, for any $\varphi\in C^\infty(\overline{\o};\hspace{0.03cm} \Cinfper(Y_2))$  satisfying \eqref{divfreeconex},
     \begin{linenomath}
     \begin{align}
     \int_{\o\times Y_2}\sigma_0(x,y)\varphi(x,y)dxdy &= \int_{\o\times Y_2}\frac{\partial u_0}{\partial x_1}(x, y_2)\varphi(x,y)dxdy \notag\\
     &\quad- \int_{0}^{1}dx_2\int_{Y_2}\l[u_0(1,x_2,y_2)\varphi(1, x_2, y) -u_0(0,x_2,y_2)\varphi(0, x_2, y) \r]dy. \notag%\label{intpartcount}
     \end{align}
     \end{linenomath}
Since for any $\varphi\in \Cinfcpt(\o;\hspace{0.03cm} \Cinfper(Y_2))$  the first two integrals are equal and bounded by a constant times $\|\varphi\|_{L^2(\Omega\times [0,1)}$, we conclude that, for any  for any $\varphi\in C^\infty(\overline{\o};\hspace{0.03cm} \Cinfper(Y_2))$  satisfying \eqref{divfreeconex},
     \begin{linenomath}
     \begin{align}
 \int_{0}^{1}dx_2\int_{Y_2}\l[u_0(1,x_2,y_2)\varphi(1, x_2, y) -u_0(0,x_2,y_2)\varphi(0, x_2, y) \r]dy=0,\notag
          \end{align}
     \end{linenomath}
 which implies that 
     \begin{linenomath}
     \begin{equation*}
     u_0(1,x_2,y_2) = u_0(0,x_2,y_2) =0 \qquad \mbox{a.e.}\hspace{0.3cm} (x_2, y_2)\in (0,1)\times Y_1.
     \end{equation*}
     \end{linenomath}
\noindent This combined with \eqref{partialu0x1} yields     
     \begin{linenomath}
     \begin{equation*}
     u_0(x_1, x_2, y_2)\in H^1_0((0,1)_{x_1}; L^2((0,1)_{x_2}\times Y_1)).
     \end{equation*}
     \end{linenomath}
Finally, an integration by parts with respect to $x_1$ of the right-hand side of \eqref{formgradchi} implies that, for any $\varphi\in C^\infty(\overline{\o}; \hspace{0.03cm} \Cinfper(Y_2))$ satisfying \eqref{divfreeconex},
    \begin{linenomath}
    \begin{align*}
    \int_{\o\times Y_2}\l( \sigma_0(x,y)-\frac{\partial u_0}{\partial x_1}(x,y_2)\r)\varphi(x,y)dxdy =0.
    \end{align*}
    \end{linenomath}
Since the orthogonal of divergence-free functions is the gradients, from the previous equality we deduce that there exists $\tilde{u}\in H^1_{\text{per}}(Y_1; L^2(\o\times Y_1))$ such that 
   \begin{equation}
   \label{expltwosca}
   \sigma_0(x,y) = \frac{\partial u_0}{\partial x_1}(x,y_2)+ \frac{\partial \tilde{u}}{\partial y_1}(x,y).
   \end{equation}
\par Let us now show  that 
    \begin{linenomath}
    \begin{equation}
    \label{V1}
    \liminf_{\varepsilon\to 0} \int_{\o}a\l(\frac{x_2}{\varepsilon}\r)\l(\frac{\partial \ue}{\partial x_1}  \r)^2dx \geq \int_{\o\times Y_2}a(y_2)\l( \frac{\partial u_0}{\partial x_1}(x,y_2)+ \frac{\partial \tilde{u}}{\partial y_1}(x,y) \r)^2dxdy. 
    \end{equation}
    \end{linenomath} 
To this end, set 
      \begin{linenomath}
      \begin{equation*}
      \sigma_\varepsilon := \frac{\partial \ue}{\partial x_1}.
      \end{equation*}
      \end{linenomath}
Since $a\in L^\infty_{\rm per}(Y_1)\subset L^2_{\rm per}(Y_1)$, there exists a sequence $a_k$ of functions in $\Cinfper(Y_1)$ such that 
        \begin{linenomath}
        \begin{equation}
        \label{V3}
        \|a-a_k\|_{L^2(Y_1)} \to 0 \quad\mbox{as $k\to\infty$},
        \end{equation}
        \end{linenomath}   
hence, by periodicity, we also have
       \begin{linenomath} 
       \begin{equation}
       \label{V4}
       \l\|a\l(\xetwo\r) - a_k\l(\xetwo\r) \r\|_{L^2(\o)} \leq C \|a-a_k\|_{L^2(Y_1)},
       \end{equation}
       \end{linenomath}
for some positive constant $C>0$. On the other hand, since $\sigma_0$ given by \eqref{expltwosca} is in $L^2(\o\times Y_2)$, there exists a sequence $\psi_n$ of functions in $\Cinfcpt(\o;\hspace{0.03cm}\Cinfper(Y_2))$ such that 
        \begin{linenomath}
        \begin{equation}
        \label{limn}
        \psi_n (x,y) \to \sigma_0(x,y)\qquad \mbox{strongly in $L^2(\o\times Y_2)$}.
        \end{equation} 
        \end{linenomath}   
From the inequality
       \begin{linenomath}
       \begin{equation*}
       \int_{\o} a\l(\xetwo\r)\l(\sigma_{\varepsilon} - \psi_n\l(x,\xe\r)\r) ^ 2 dx\geq 0,
       \end{equation*}
       \end{linenomath} 
we get
        \begin{linenomath}
        \begin{align}
        \int_{\o}a\l(\xetwo\r)\sigma_\varepsilon^2dx&\geq 2\int_{\o}a\l(\xetwo\r)\sigma_\varepsilon\psi_n\l(x,\xe\r)dx -\int_{\o}a\l(\xetwo\r)\psi_n^2\l(x,\xe\r)dx\notag\\
        &=2\int_{\o}\l(a\l(\xetwo\r)-a_k\l(\xetwo\r) \r)\sigma_\varepsilon\psi_n\l(x,\xe\r)dx + 2\int_{\o}a_k\l(\xetwo\r)\sigma_\varepsilon\psi_n\l(x,\xe\r)dx\notag\\
        &\quad-\int_{\o}a\l(\xetwo\r)\psi_n^2\l(x,\xe\r)dx.\label{V5}
        \end{align}
        \end{linenomath}
In view of \eqref{V4}, the first integral on the right-hand side of \eqref{V5} can be estimated as
      \begin{linenomath}
      \begin{align*}
      \l|\int_{\o}\l(a\l(\xetwo\r)-a_k\l(\xetwo\r) \r)\sigma_\varepsilon\psi_n\l(x,\xe\r)dx \r| &\leq C \|a-a_k\|_{L^2(Y_1)}\|\psi_n\|_{L^\infty(\o)}\|\sigma_\varepsilon\|_{L^2(\o)}\\
      &\leq C \|a-a_k\|_{L^2(Y_1)}.
      \end{align*}
      \end{linenomath}  
Hence, passing to the limit as $ \varepsilon\to 0$ in \eqref{V5} with the help of \eqref{convpartder} leads to
       \begin{linenomath}
       \begin{align*}
       \liminf_{\varepsilon\to 0} \int_{\o} a\l(\xetwo\r)\sigma^2_\varepsilon dx&\geq- C \|a-a_k\|_{L^2(Y_1)}+ 2\lim_{\varepsilon\to 0} \int_{\o}a_k\l(\xetwo\r)\sigma_\varepsilon\psi_n\l(x,\xe\r)dx \\
       &\quad -\lim_{\varepsilon\to 0}\int_{\o}a\l(\xetwo\r)\psi_n^2\l(x,\xe\r)dxdy\\
       &=2\int_{\o\times Y_2}a_k(y_2)\sigma_0(x,y)\psi_n(x,y)dxdy - C\|a-a_k\|_{L^2(Y_1)}\\
       &\quad-\int_{\o\times Y_2}a(y_2)\psi_n^2(x,y)dxdy.
       \end{align*}
      \end{linenomath}
Thanks to \eqref{V3},  we take the limit as $k\to\infty$ in the previous inequality and we obtain
           \begin{linenomath}   
           \begin{align*}
           \liminf_{\varepsilon\to 0} \int_{\o} a\l(\xetwo\r)\sigma^2_\varepsilon dx &\geq 2\int_{\o\times Y_2}a(y_2)\sigma_0(x,y)\psi_n(x,y)dxdy -\int_{\o\times Y_2}a(y_2)\psi_n^2(x,y)dxdy,
           \end{align*}
           \end{linenomath}
so that in view of \eqref{limn}, passing to the limit as $n\to\infty$ leads to
       \begin{linenomath}
       \begin{align*}
       \liminf_{\varepsilon\to 0} \int_{\o} a\l(\xetwo\r)\sigma^2_\varepsilon dx %&\geq 2\int_{\o\times Y_2}a(y_2)\xi_0(x,y)\xi_0(x,y)dxdy-\int_{\o\times Y_2}a(y_2)\xi_0^2(x,y)dxdy\\
       &\geq \int_{\o\times Y_2}a(y_2)\sigma_0^2(x,y)dxdy.
       \end{align*}
       \end{linenomath}
This combined with \eqref{expltwosca} proves \eqref{V1}.  \par By \eqref{u10indpey1},  we already know that $u_0$ does not depend on $y_1$. In view of the periodicity of $\tilde{u}$ with respect to $y_1$, an application of Jensen's inequality leads us to 
     \begin{linenomath}
     \begin{align}
    \int_{\o\times Y_2} a(y_2)&\l(\frac{\partial u_0}{\partial x_1}(x,y_2) +\frac{\partial \tilde{u}}{\partial y_1} (x,y)\r)^2dxdy\notag\\
     &=\int_{\o}dx\int_{Y_1}a(y_2)dy_2\int_{Y_1}\l(\frac{\partial u_0}{\partial x_1}(x,y_2) +\frac{\partial \tilde{u}}{\partial y_1} (x,y)\r)^2dy_1\notag\\
     &\geq\int_{\o}dx\int_{Y_1}a(y_2)dy_2\l(\int_{Y_1} \l[\frac{\partial u_0}{\partial x_1}(x,y_2) +\frac{\partial \tilde{u}}{\partial y_1} (x,y)\r] dy_1 \r)^2\notag\\
      &=\int_{\o}dx\int_{Y_1}a(y_2)\l(\frac{\partial u_0}{\partial x_1}\r)^2(x,y_2)dy_2. \notag%\label{jenine}
     \end{align}
     \end{linenomath}
This combined with \eqref{V1} implies that
    \begin{linenomath}
    \begin{align}
    \liminf_{\varepsilon\to 0}\int_{\o}a\l(\frac{x_2}{\varepsilon}\r)\l(\frac{\partial\ue}{\partial x_1}\r)^2dx &\geq\int_{\o}dx\int_{Y_1}a(y_2)\l(\frac{\partial u_0}{\partial x_1}\r)^2(x,y_2)dy_2.\label{V6}
    \end{align}
    \end{linenomath}
\par Now, we extend the functions in $L^2(\o)$ by zero with respect to $x_1$ outside $(0,1)$ so that functions in $H^1_0((0,1)_{x_1};L^2(0,1)_{x_2})$ can be regarded as functions in $H^1(\mathbb{R}_{x_1}; L^2(0,1)_{x_2})$. Due to the weak $L^2$-lower semi-continuity of $\|\ue\|_{L^2(\o)}$ along with \eqref{V6}, we have
   \begin{linenomath}
   \begin{align}
   \liminf_{\varepsilon\to 0} \fe(\ue) 
   &\geq \int_{0}^{1}dx_2\int_{Y_1}dy_2\int_{\mathbb{R}}\l[a(y_2)\l(\frac{\partial u_0}{\partial x_1}\r)^2(x_1, x_2,y_2)+ |u_0|^2(x_1, x_2,y_2)\r]dx_1.\label{c7}
    \end{align}
    \end{linenomath}
We minimize the right-hand side with respect to $u_0(x_1, x_2, y_2)\in H^1(\mathbb{R}_{x_1};\hspace{0.03cm} L^2((0,1)_{x_2}\times Y_1))$ satisfying \eqref{c3} where the weak limit $u$ of $\ue$ in $L^2(\o)$ is fixed. The minimizer, still denoted  by $u_0$, satisfies the Euler equation
     \begin{linenomath}
     \begin{equation*}
      \int_{0}^{1}dx_2\int_{Y_1}dy_2\int_{\mathbb{R}}\l[a(y_2)\frac{\partial u_0}{\partial x_1}(x_1, x_2,y_2)\frac{\partial v}{\partial x_1}(x_1, x_2,y_2)+ u_0(x_1, x_2,y_2)v(x_1, x_2,y_2)\r]dx_1=0
     \end{equation*}
     \end{linenomath}
for any $v(x_1, x_2, y_2)\in H^1(\mathbb{R}_{x_1};\hspace{0.03cm} L^2((0,1)_{x_2}\times Y_1)) $ such that $\int_{Y_1} v(x,y_2)dy_2=0$. Then, there exists $b(x_1,x_2)\in H^{-1}(\mathbb{R}_{x_1};\hspace{0.03cm} L^2(\mathbb{R})_{x_2})$ independent of $y_2$ such that in distributions sense with respect to the variable $x_1$,
     \begin{equation}
     \label{c4}
     -a(y_2)\frac{\partial^2 u_0}{\partial x_1^2}(x_1, x_2 ,y_2) + u_0(x_1, x_2,y_2) = b(x_1, x_2) \quad \mbox{in}\hspace{0.1cm}\mathscr{D}'(\R)\hspace{0.3cm}\mbox{a.e.}\hspace{0.1cm}(x_2, y_2)\in (0,1)\times Y_1.
     \end{equation}
Taking the Fourier transform $\fourtra_2$ on $L^2(\R)$ of parameter $\lambda_1$ with respect to the variables $x_1$, the equation \eqref{c4} becomes
    \begin{equation}
    \label{c5}
    \fourtra_2(u_0)(\lambda_1,x_2, y_2) = \frac{\fourtra_2(b)(\lambda_1, x_2)}{4\pi^2a(y_2)\lambda_1^2+1} \qquad\mbox{a.e.}\hspace{0.1cm} (\lambda_1, x_2, y_2)\in \R\times (0,1)\times Y_1.
    \end{equation} 
Note that  \eqref{c5} proves in particular that the two-scale limit $u_0$ does depend on the variable $y_2$, since its Fourier transform with respect to the variable $x_1$ depends on $y_2$ through the function $a(y_2)$.
\par In light of the definition \eqref{defk0} of $\hat{k}_0$ and due to \eqref{c3}, integrating \eqref{c5} with respect to $y_2\in Y_1$ yields 
    \begin{equation}
    \label{c6}
    \fourtra_2(u)(\lambda_1, x_2) =  \hat{k}_0(\lambda_1)\fourtra_2 (b)(\lambda_1, x_2)\qquad\mbox{a.e.}\hspace{0.1cm} (\lambda_1, x_2)\in \R\times (0,1).    
    \end{equation} 
By using Plancherel's identity with respect to the variable $x_1$  in the right-hand side of \eqref{c7} and in view of  \eqref{c5} and \eqref{c6}, we obtain 
     \begin{linenomath}
     \begin{align}
     \liminf_{\varepsilon\to 0} \fe(\ue)&\geq \int_{0}^{1}dx_2\int_{Y_1}dy_2\int_{\mathbb{R}} (4\pi^2a(y_2)\lambda_1^2 +1)|\fourtra_2(u_0)(\lambda_1,x_2, y_2)|^2d\lambda_1\notag\\
     &=\int_{0}^{1}dx_2\int_{\mathbb{R}} \frac{1}{\hat{k}_0(\lambda_1)} |\fourtra_2(u)(\lambda_1, x_2)|^2d\lambda_1,\notag%\label{c11}
     \end{align} 
     \end{linenomath}
which proves the $\Gamma$-$\liminf$ inequality.

\smallskip 
\noindent\textit{Step 2}- $\Gamma$-$\limsup$ inequality.\\
For the proof of the $\Gamma$-$\limsup$ inequality, we need the following lemma whose proof will be given later.
  \begin{lemma}
  \label{lemma:solu0}
  Let $u\in\Cinfcpt(\o)$. For fixed $x_2\in (0,1)$ and $y_2\in Y_1$, let $b(\cdot, x_2)$ be the distribution (parameterized by $x_2$) defined by 
        \begin{linenomath}
        \begin{equation}
        \label{fourukb}
        \fourtra_2(b)(\lambda_1, x_2):=\frac{1}{\hat{k}_0(\lambda_1)}
        \fourtra_2(u)(\lambda_1, x_2),
        \end{equation}
        \end{linenomath}
  where $u(\cdot, x_2)$ is extended by zero outside $(0,1)$. Let $u_0(\cdot, x_2, y_2)$ be the unique solution to problem 
    \begin{linenomath}
     \begin{align}
     \label{stuliopb}
     \begin{dcases}
     -a(y_2)\frac{\partial^2 u_0}{\partial x_1^2}(x_1,x_2,y_2) + u_0(x_1,x_2,y_2) = b(x_1,x_2), & x_1\in (0,1),\\
     u_0(0, x_2, y_2) = u_0(1,x_2,y_2)=0,&
     \end{dcases}
     \end{align}
     \end{linenomath}
  with $a(y_2)$ given by \eqref{acountex}.
  Then $b(x_1, x_2)$ is in $C([0,1]_{x_2};\hspace{0.03cm}L^2(0,1)_{x_1})$ and $u_0(x_1, x_2, y_2)$ is in $C^1([0,1]^2;\hspace{0.03cm}L^\infty_{\rm per}(Y_1))$.
  \end{lemma}
 Let $u\in\Cinfcpt(\o)$. Thanks to Lemma \ref{lemma:solu0}, there exists a unique solution 
      \begin{linenomath}
      \begin{equation}
      \label{solu0SLpb}
      u_0(x_1, x_2, y_2)\in C^1([0,1]^2; \hspace{0.03cm}L^\infty_{\rm per}(Y_1))
      \end{equation}
      \end{linenomath}
to the problem \eqref{stuliopb}. 
Taking the Fourier transform $\fourtra_2$ on $L^2(\R)$ of parameter $\lambda_1$ with respect to $x_1$ of the equation in \eqref{stuliopb}  and taking into account \eqref{fourukb}, we get 
     \begin{linenomath}
     \begin{equation}
     \label{fourtra}
     \fourtra_2(u_0)(\lambda_1, x_2, y_2) = \frac{\fourtra_2(u)(\lambda_1, x_2) }{(4\pi^2a(y_2)\lambda_1^2+1)\hat{k}_0(\lambda_1)}\qquad \mbox{for}\hspace{0.3cm}(\lambda_1, x_2, y_2)\in\R\times [0,1]\times Y_1,
     \end{equation}
     \end{linenomath} 
where  $u_0(\cdot, x_2,y_2)$ and $u(\cdot, x_2)$  are extended by zero outside $(0,1)$. Integrating \eqref{fourtra} over $y_2\in Y_1$, we obtain
      \begin{linenomath}
     \begin{equation}
     \label{weaklimicounteex}
     u(x_1,x_2) = \int_{Y_1} u_0(x_1, x_2,y_2)dy_2 \qquad \mbox{for}\hspace{0.1cm}(x_1, x_2)\in\mathbb{R}\times (0,1).
     \end{equation}
     \end{linenomath}
Let $\{\ue\}_\varepsilon$ be the sequence in $L^2(\o)$ defined by
     \begin{linenomath}
     \begin{equation*}
     \label{defuesp}
     u_\varepsilon(x_1, x_2) := u_0\l(x_1, x_2,\xetwo\r).
     \end{equation*}
     \end{linenomath}
Recall that rapidly oscillating $Y_1$-periodic function $\ue$ weakly converges  in $L^2(\o)$ to the mean value of $\ue$ over $Y_1$. This combined with \eqref{weaklimicounteex} implies that  $\ue$ weakly converges in $L^2(\o)$ to $u$. In other words,
     \begin{linenomath}
     \begin{equation*}
     \ue \rightharpoonup u \quad \mbox{weakly in}\hspace{0.2cm} L^2(\o).
     \end{equation*}
     \end{linenomath}
Due to \eqref{solu0SLpb}, we can apply  \cite[Lemma 5.5]{A92} so that $u_0(x_1, x_2, y_2)$ and $\displaystyle{\partial u_0\over \partial x_1}$ are an admissible test function for the two-scale convergence. Hence,
    \begin{linenomath}
    \begin{align}
    \lim_{\varepsilon\to 0}\fe(\ue) &=\lim_{\varepsilon\to 0} \int_{\o}\l[a\l(\xetwo\r)\l(\frac{\partial u_0 }{\partial x_1} \r)^2\l(x_1, x_2,\frac{x_2}{\varepsilon}\r)  +\l|u_0\l(x_1, x_2,\xetwo\r)\r|^2 \r]dx\notag\\
    &= \int_{\o}dx\int_{Y_1}\l[ a(y_2)\l(\frac{\partial u_0}{\partial x_1} \r)^2(x_1, x_2, y_2)   +\l|u_0(x_1, x_2, y_2)\r|^2\r]dy_2\notag\\\
    &=\int_{0}^{1}dx_2\int_{Y_1}dy_2\int_{\mathbb{R}}\l[ a(y_2)\l(\frac{\partial u_0}{\partial x_1} \r)^2(x_1, x_2,y_2)+\l|u_0(x_1, x_2,y_2)\r|^2\r]dx_1, \label{limsup1}
    \end{align}
    \end{linenomath}
where the function $x_1\mapsto u_0(x_1,\cdot,\cdot)$ is extended by zero  outside $(0,1)$. In view of the definition \eqref{defk0} of $\hat{k}_0$ and due to \eqref{fourtra},  the Plancherel identity with respect to the variable $x_1$ and the Fubini theorem yield 
      \begin{linenomath}
      \begin{align*}
       \int_{0}^{1}dx_2\int_{Y_1}dy_2\int_{\mathbb{R}}&\l[ a(y_2)\l(\frac{\partial u_0}{\partial x_1} \r)^2(x_1, x_2,y_2)+\l|u_0(x_1, x_2, y_2)\r|^2\r]dx_1\notag\\
      &= \int_{0}^{1}dx_2\int_{Y_1}dy_2\int_{\mathbb{R}}(4\pi^2a(y_2)\lambda^2_1+1)|\fourtra_2(u_0)(\lambda_1, x_2,y_2)|^2d\lambda_1\notag\\
      &=\int_{0}^{1}dx_2\int_{\mathbb{R}} \frac{1}{\hat{k}_0(\lambda_1)}|\fourtra_2(u)(\lambda_1,x_2)|^2d\lambda_1.\label{limsup2}
      \end{align*} 
      \end{linenomath}
This together with \eqref{limsup1} implies that, for $u\in\Cinfcpt(\o)$,
      \begin{linenomath}
      \begin{equation*}
      \lim_{\varepsilon\to 0} \fe(\ue) = \int_{0}^{1}dx_2\int_{\mathbb{R}} \frac{1}{\hat{k}_0(\lambda_1)}|\fourtra_2(u)(\lambda_1,x_2)|^2d\lambda_1,
      \end{equation*}
      \end{linenomath}
which proves the $\Gamma$-$\limsup$ inequality on $\Cinfcpt(\o)$.
\par Now, let us extend the previous result to any $u\in H^1_0((0,1)_{x_1};\hspace{0.03cm}L^2(0,1)_{x_2})$. To this end, we use a density argument (see \textit{e.g.} \cite[Remark 2.8]{Brai06}). Recall that the weak topology of $L^2(\o)$ is metrizable on the closed balls of $L^2(\o)$.  Fix $n\in\mathbb{N}$ and denote $d_{B_n}$ any metric inducing the $L^2(\o)$-weak topology on the ball $B_n$ centered on $0$ and of radius $n$. Then,  $H^1_0((0,1)_{x_1}; L^2(0,1)_{x_2})$ can be regarded as a subspace of $L^2(\o)$ endowed with the metric $d_{B_n}$. On the other hand, $H^1_0((0,1)_{x_1}; L^2(0,1)_{x_2})$ is a Hilbert space endowed with the norm 
     \begin{linenomath}
     \begin{equation*}
     \|u\|_{H^1_0((0,1)_{x_1}; L^2(0,1)_{x_2})} :=\l( \l\|\frac{\partial u}{\partial x_1}  \r\|^2_{L^2(\o)} + \|u\|^2_{L^2(\o)}\r)^{1/2}.
     \end{equation*}
     \end{linenomath}
The associated metric $d_{H^1_0}$ on $H^1_0((0,1)_{x_1}; L^2(0,1)_{x_2})$   induces a topology which is not weaker than that induced by $d_{B_n}$, \textit{i.e.} 
      \begin{linenomath}
      \begin{equation}
      \label{weaktopology}
      d_{H^1_0}(u_k, u)\to 0\hspace{0.4cm}\mbox{implies }\hspace{0.4cm}d_{B_n}(u_k, u)\to 0.
      \end{equation}
      \end{linenomath}
Recall that $\Cinfcpt(\o)$ is a dense subspace of $H^1_0((0,1)_{x_1}; L^2(0,1)_{x_2})$ for the metric $d_{H^1_0}$ and that the $\Gamma$-$\limsup$ inequality holds on  $\Cinfcpt(\o)$ for the $L^2(\o)$-weak topology, \textit{i.e.}  for any $u\in\Cinfcpt(\o)$,
    \begin{linenomath}
    \begin{equation}
    \label{gammalimsup}
    \Gamma\text{-}\limsup_{\varepsilon\to 0}\fe(u)\leq \mathscr{F}(u).
    \end{equation}
    \end{linenomath}
A direct computation of $\hat{k}_0$, given by \eqref{defk0}, shows that 
              \begin{linenomath}
              \begin{align*}
              \hat{k}_0(\lambda_1) &= \frac{c_\theta4\pi^2\lambda_1^2+1}{(4\pi^2\lambda_1^2+ 1)(c4\pi^2\lambda_1^2+ 1)}, 
              \end{align*}
              \end{linenomath}  
        which implies  that
              \begin{linenomath}
              \begin{align}
              \label{four1}
              \frac{1}{\hat{k}_0(\lambda_1) } =\frac{c}{c_\theta}4\pi^2\lambda_1^2 + f(\lambda_1) + \alpha,
               % + \frac{c^2\theta +1-\theta}{c_\theta^2}+\frac{(c-1)^2\theta(\theta-1)}{c^2_\theta}\frac{1}{c_\theta4\pi^2\lambda_1^2 + 1}.
              \end{align}
              \end{linenomath}
where $f(\lambda_1)$ and $\alpha$ are given by \eqref{alphaf}. Hence, there exists a positive constant $C$ such that 
    \begin{linenomath}
    \begin{equation}
    \label{boundednessk_0}
    \frac{1}{\hat{k}_0(\lambda_1)} \leq C(4\pi^2\lambda_1^2 + 1).
    \end{equation}
    \end{linenomath} 
This combined with the Plancherel identity 
yields 
    \begin{linenomath}
    \begin{align}
    \mathscr{F}(u)&\leq C\int_{0}^{1}dx_2\int_{\R} (4\pi^2\lambda_1^2 + 1) |\fourtra_2(u)(\lambda_1, x_2)|^2d\lambda_1\notag\\
    &= C\int_{0}^{1}dx_2\int_{\R}\l[ \l( \frac{\partial u}{\partial x_1} \r)^2(x_1, x_2)+ |u(x_1, x_2)|^2\r]dx_1\notag\\
    &=C \|u\|^2_{H^1_0((0,1)_{x_1}; L^2(0,1)_{x_2})},\label{Fcontiuous}
    \end{align}
    \end{linenomath}
where $u(\cdot, x_2)$ is extended by zero outside $(0,1)$. Since $\mathscr{F}$ is a non-negative quadratic form, from \eqref{Fcontiuous} we conclude that  $\mathscr{F}$ is continuous  with respect to the metric $d_{H^1_0}$. 
\par Now, take $u\in H^1_0((0,1)_{x_1}; L^2(0,1)_{x_2})$. By density, there exists a sequence $u_k$ in $\Cinfcpt(\o)$ such that
      \begin{linenomath}
      \begin{equation}
      \label{Cinfdense}
      d_{H^1_0} (u_k, u)\to 0\qquad\mbox{as}\hspace{0.3cm} k\to\infty.
      \end{equation}
      \end{linenomath}
In particular, due to \eqref{weaktopology}, we also have that $d_{B_n} (u_k, u)\to 0$ as $k\to\infty$.
In view of the weakly lower semi-continuity of $\Gamma$-$\limsup$ and the continuity of $\mathscr{F}$, we deduce from \eqref{gammalimsup} that 
     \begin{linenomath}
     \begin{align*}
     \Gamma\text{-}\limsup_{\varepsilon\to 0}\fe(u)&\leq \liminf_{k\to\infty} (\Gamma\text{-}\limsup_{\varepsilon\to 0}\fe(u_k) )\\
     &\leq  \liminf_{k\to\infty}\mathscr{F}(u_k)\\
     &= \mathscr{F}(u),
     \end{align*}
     \end{linenomath} 
which proves the $\Gamma$-$\limsup$ inequality in $B_n$. Since for any $u\in H^1_0((0,1)_{x_1}; L^2(0,1)_{x_2})$ the sequence $u_k$ of functions in $\Cinfcpt(\o)$ satisfying \eqref{Cinfdense} belongs to some ball $B_n$ of $L^2(\o)$, as well as its limit, the $\Gamma$-$\limsup$ property holds true for the sequence $\fe$ on $H^1_0((0,1)_{x_1}; L^2(0,1)_{x_2})$, which concludes the proof of $\Gamma$-$\limsup$ inequality.

\smallskip\noindent
\noindent\textit{Step 3} - Alternative expression of $\Gamma$-limit.\\
The proof of the equality between the two expressions of the $\Gamma$-limit $\mathscr{F}$ relies on the following lemma whose proof will be given later.
     \begin{lemma}
     \label{firstlemma}
     Let $h \in L^2(\mathbb{R})$   and $u\in L^1(\mathbb{R})\cap L^2(\mathbb{R})$. Then, $h\ast u\in L^2(\R)$ and 
          \begin{linenomath}
          \begin{equation}
          \label{conprop}
          \fourtra_2 (h\ast u) = \fourtra_2(h)\fourtra_2(u)\quad \mbox{a.e. in $\mathbb{R}$}.
          \end{equation}
          \end{linenomath}
     \end{lemma}
 By applying Plancherel's identity with respect to $x_1$,  for any $u\in H^1_0(\R_{x_1}; L^2(0,1)_{x_2})$ extended by zero with respect to the variable $x_1$ outside $(0,1)$, we get
       \begin{linenomath}
       \begin{align}
       \int_{\mathbb{R}}&\l|\sqrt{\alpha}u(x_1,x_2) + (h\ast_1u)(x_1,x_2)\r|^2dx_1 \notag\\
       &=\int_{\mathbb{R}}\l|\sqrt{\alpha}\fourtra_2(u)(\lambda_1,x_2) + \fourtra_2(h\ast_1u)(\lambda_1,x_2)\r|^2d\lambda_1 \notag\\
        &=\int_{\mathbb{R}} \l[\alpha \l|\fourtra_2 (u)(\lambda_1,x_2)\r|^2  + 2\sqrt{\alpha}{\rm Re}\l(\fourtra_2(u)(\lambda_1, x_2) \overline{\fourtra_2(h\ast_1u)}(\lambda_1,x_2)\r) +\l|\fourtra_2(h\ast_1u)(\lambda_1,x_2)\r|^2\r]d\lambda_1.\label{counex1}
       \end{align}
       \end{linenomath} 
  Recall that the Fourier transform of $h$, given by \eqref{four2}, is real. From \eqref{counex1},  an application of Lemma \ref{firstlemma}  leads us to
       \begin{linenomath}
       \begin{align}
       \int_{\mathbb{R}} &\l[\alpha \l|\fourtra_2 (u)(\lambda_1,x_2)\r|^2  + 2\sqrt{\alpha}{\rm Re}\l(\fourtra_2(u)(\lambda_1, x_2) \overline{\fourtra_2(h\ast_1u)}(\lambda_1,x_2)\r) +\l|\fourtra_2(h\ast_1u)(\lambda_1,x_2)\r|^2\r]d\lambda_1\notag\\
       &= \int_{\mathbb{R}} \l[\alpha+2\sqrt{\alpha}\fourtra_2(h)(\lambda_1) + \l(\fourtra_2(h)(\lambda_1)\r)^2\r]\l|\fourtra_2(u)(\lambda_1,x_2)\r|^2d\lambda_1\notag\\
       &= \int_{\mathbb{R}} \l[\sqrt{\alpha}+ \fourtra_2(h)(\lambda_1)\r]^2|\fourtra_2(u)(\lambda_1,x_2)|^2d\lambda_1\notag\\
       &= \int_{\mathbb{R}} \l[\alpha+ f(\lambda_1)\r]|\fourtra_2(u)(\lambda_1,x_2)|^2d\lambda_1. \label{V7}
       \end{align}
       \end{linenomath}
On the other hand, by applying  Plancherel's identity with respect to $x_1$, we obtain 
    \begin{linenomath}
    \begin{equation*}
    \int_{\R}\frac{c}{c_\theta}4\pi^2\lambda_1^2|\fourtra_2(u)(\lambda_1,x_2)|^2d\lambda_1 = \int_{\R}\frac{c}{c_\theta}\l(\frac{\partial u}{\partial x_1}\r)^2(x_1,x_2)dx_1.
    \end{equation*}
    \end{linenomath}
In view of the expansion of $1/\hat{k}_0(\lambda_1)$ given by \eqref{four1}, the previous equality combined with \eqref{counex1} and \eqref{V7} implies that, for $u\in H^1_0((0,1)_{x_1};\hspace{0.1cm}L^2(0,1)_{x_2})$ extended by zero with respect to $x_1$ outside $(0,1)$,
     \begin{linenomath}
     \begin{align*}
     \int_{0}^{1}dx_2\int_{\mathbb{R}}& \frac{1}{\hat{k}_0(\lambda_1)} |\fourtra_2(u)(\lambda_1,x_2)|^2d\lambda_1 \\
     &=\int_{0}^{1}dx_2 \int_{\mathbb{R}}\l\{\frac{c}{c_\theta}\l(\frac{\partial u}{\partial x_1}\r)^2(x_1,x_2)+[\sqrt{\alpha}u(x_1,x_2) + (h\ast_1u)(x_1,x_2)]^2\r\}dx_1,
     \end{align*}
     \end{linenomath}
which concludes the proof.
\end{proof}

    \begin{proof}[Proof of Lemma \ref{lemma:solu0}]
    In view of  \eqref{four1}, the equality \eqref{fourukb} becomes
        \begin{linenomath}
        \begin{align}
        \fourtra_2(b)(\lambda_1, x_2 ) &= \l(\frac{c}{c_\theta}4\pi^2\lambda_1^2+\alpha + f(\lambda_1)\r)\fourtra_2(u)(\lambda_1, x_2)\notag\\
        &= \fourtra_2 \l(-\frac{c}{c_\theta}\frac{\partial^2 u}{\partial x_1^2} +\alpha u\r)(\lambda_1, x_2) + f(\lambda_1)\fourtra_2(u)(\lambda_1,x_2).\label{fourblemma}
        \end{align}
        \end{linenomath}
        Since
              \begin{linenomath}
              \begin{equation*}
              \label{fCL1}
              f(\lambda_1)= \frac{(c-1)^2\theta(\theta-1)}{c^2_\theta}\frac{1}{c_\theta4\pi^2\lambda_1^2 + 1}= O(\lambda_1^{-2})\in C_0(\mathbb{R})\cap L^1(\mathbb{R}),
              \end{equation*}
              \end{linenomath}
         the right-hand side of \eqref{fourblemma} belongs to $L^2(\R)$ with respect to $\lambda_1$, which implies that
         \begin{linenomath}
         \begin{equation*}
         \fourtra_2(b)(\cdot, x_2)\in L^2(\R).
         \end{equation*}
         \end{linenomath}
    Applying the Plancherel identity, we obtain that $b(\cdot, x_2)\in L^2(\R)$ with respect to $x_1$. Since $u(\cdot, x_2)$ is extended by zero outside $(0,1)$, $b(\cdot, x_2)$ is also equal to zero outside $(0,1)$ so that
            \begin{linenomath}
            \begin{equation}
            \label{bL2}
            b(\cdot, x_2)\in L^2(0,1).
            \end{equation}
            \end{linenomath}
     \par Let us show that $b(x_1,\cdot)$ is a continuous function with respect to $x_2\in [0,1]$. Recall that the continuity of $x_2\in [0,1]\mapsto b(x_1, x_2) \in L^2(0,1)_{x_1}$ is equivalent to 
           \begin{linenomath}
           \begin{equation*}
           \lim_{t\to 0} \|b(\cdot, x_2+t) - b(\cdot, x_2)\|_{L^2(0,1)_{x_1}} =0.
           \end{equation*}
           \end{linenomath}
    Thanks to Plancherel's identity, we infer from \eqref{fourukb} that
        \begin{linenomath}
        \begin{align*}
        \|b(\cdot, x_2+t) - b(\cdot, x_2)\|^2_{L^2(0,1)_{x_1}} &= \|\fourtra_2(b)(\cdot, x_2+t) - \fourtra_2(b)(\cdot, x_2)\|^2_{L^2(\R)_{\lambda_1}}\\
        &=\int_{\R}\l|\frac{1}{\hat{k}_0(\lambda_1)}\l[\fourtra_2(u)(\lambda_1, x_2+t) - \fourtra_2(u)(\lambda_1, x_2)\r] \r|^2d\lambda_1.
        \end{align*}
        \end{linenomath}
    In view of\eqref{boundednessk_0} and thanks to the Plancherel identity, we obtain  
       \begin{linenomath}
       \begin{align}
       \|b(\cdot, x_2+t) - b(\cdot, x_2)\|^2_{L^2(0,1)_{x_1}} &\leq C^2
       \int_{\R}\l|(4\pi^2\lambda_1^2+1) (\fourtra_2(u)(\lambda_1, x_2+t) - \fourtra_2(u)(\lambda_1, x_2)) 
       \r|^2d\lambda_1\notag\\
       &\leq C^2\l\|\fourtra_2 \l(\frac{\partial u}{\partial x_1}\r) (\cdot, x_2+t ) -\fourtra_2\l(\frac{\partial u}{\partial x_1}\r) (\cdot, x_2 )  \r\|^2_{L^2(0,1)_{\lambda_1}} \notag\\
       &\quad+ C^2 \|\fourtra_2(u)(\cdot, x_2+t)-\fourtra_2(u)(\cdot, x_2)\|^2_{L^2(0,1)_{\lambda_1}} \notag\\
       &= C^2\l\|\frac{\partial u}{\partial x_1}(\cdot, x_2+t ) -\frac{\partial u}{\partial x_1} (\cdot, x_2 )  \r\|^2_{L^2(0,1)_{x_1}}\notag \\
       &\quad +C^2  \|u(\cdot, x_2+t)-u(\cdot, x_2)\|^2_{L^2(0,1)_{x_1}}. \notag%\label{continuity}
       \end{align}
       \end{linenomath}
     By the Lebesgue dominated convergence theorem and since $u\in\Cinfcpt([0,1]^2)$, from the previous inequality we conclude that the map $x_2\in [0,1]\mapsto b(x_1, x_2)\in L^2(0,1)_{x_1}$ is continuous. Hence, 
        \begin{linenomath}
        \begin{equation}
        \label{b}
        b(x_1, x_2)\in C([0,1]_{x_2}; \hspace{0.03cm} L^2(0,1)_{x_1}).
        \end{equation}
        \end{linenomath}
    \par To conclude the proof, it remains to show the regularity of $u_0$. Note that \eqref{stuliopb} is a Sturm-Liouville problem with constant coefficient with respect to $x_1$, since $x_2\in (0,1)$ and $y_2\in Y_1$ play the role of parameters. By \eqref{bL2}, we already know that $b(\cdot, x_2)\in L^2(0,1)$, so that thanks to a classical regularity result (see \textit{e.g.} \cite{Bre10} pp. 223-224), the problem \eqref{stuliopb} admits a unique solution $u_0(\cdot, x_2, y_2)$ in $H^2(0,1)$. Since $H^2(0,1)$ is embedding into $C^1([0,1])$, %see pag 217 Brezis
     we have 
         \begin{linenomath}
         \begin{equation*}
         u_0(\cdot, x_2, y_2)\in C^1([0,1])\qquad\mbox{a.e.}\hspace{0.3cm} (x_2, y_2)\in (0,1)\times Y_1. 
         \end{equation*}
         \end{linenomath}
      On the other hand, the solution $u_0(x_1, x_2, y_2)$ to the Sturm-Liouville problem \eqref{stuliopb} is explicitly given by  
         \begin{linenomath}
         \begin{equation}
         \label{explu0}
         u_0(x_1, x_2, y_2) :=\int_{0}^{1} G_{y_2}(x_1, s) b(s, x_2)ds,
         \end{equation}
         \end{linenomath}
    where $b(x_1, x_2)$ is defined by \eqref{fourukb} and \eqref{b} and  the kernel $G_{y_2}(x_1, s)$ is given by 
         \begin{linenomath}
         \begin{equation*}
         G_{y_2}(x_1, s) :=\frac{1}{\sqrt{a(y_2)}\sinh\l(\frac{1}{\sqrt{a(y_2)}}\r) }\sinh\l(\frac{x_1\wedge s}{\sqrt{a(y_2)}}\r)\sinh\l(\frac{1-x_1\vee s}{\sqrt{a(y_2)}}\r). 
         \end{equation*}
         \end{linenomath}
    This combined with \eqref{b} and \eqref{explu0} proves that  
         \begin{linenomath}
         \begin{equation*}
         \label{u0isregular}
         u_0(x_1, x_2, y_2) \in C^1([0,1]^2, L^\infty_{\rm per}(Y_1)),
         \end{equation*}
         \end{linenomath}
    which concludes the proof.
    \end{proof}
We prove now the Lemma \ref{firstlemma} that we used in Step $3$ of Proposition \ref{prop1}.    
     \begin{proof}[Proof of Lemma \ref{firstlemma}]
     By the convolution property of the Fourier transform on $L^2(\R)$, we have 
            \begin{linenomath}
            \begin{equation}
            \label{fou1}
            h\ast u  =  \overline{\fourtra_2}(\fourtra_2(h)) \ast \overline{\fourtra_2}(\fourtra_2(h)) =  \overline{\fourtra_1}(\fourtra_2(h)\fourtra_2(u)),
            \end{equation}
            \end{linenomath}
     where $\overline{\fourtra_i}$ denotes the conjugate Fourier transform for $i=1,2$.
     On the other hand, since $u\in L^1(\mathbb{R})\cap L^2(\R)$ and due to Riemann-Lebesgue's lemma , we deduce that $\fourtra_2(u)=\fourtra_1(u)\in C_0(\mathbb{R})\cap L^2(\R)$. This combined with $\fourtra_2(h)\in L^2(\R)$ implies that
         \begin{linenomath}
         \begin{equation*}
         \fourtra_2(h)\fourtra_2(u)=\fourtra_2(h)\fourtra_1(u)\in L^2(\mathbb{R})\cap L^1(\R).
         \end{equation*}
         \end{linenomath}
       Since $\overline{\fourtra_1} = \overline{\fourtra_2}$  on $L^1(\R)\cap L^2(\R)$, from  \eqref{fou1}  we deduce that 
          \begin{linenomath}
          \begin{equation*}
          h\ast u= \overline{\fourtra_2}(\fourtra_2(h)\fourtra_2(u))\in L^2(\mathbb{R}),
          \end{equation*}
          \end{linenomath}
     which yields  \eqref{conprop}. This concludes the proof.
     \end{proof}   
    \begin{remark}
    \label{rmk:BDrepresentation} 
    Thanks to the Beurling-Deny theory of Dirichlet forms {\rm   \cite{BD58}},  Mosco {\rm  \cite[Theorem 4.1.2]{M94} has proved that} the $\Gamma$-limit $F$ of a family of Markovian form for the $L^2(\o)$-strong topology is a Dirichlet form which can be split into a sum of three forms: a strongly local form $F_d$, a local form and nonlocal one. More precisely, for $u\in L^2(\o)$ with $F(u)<\infty$, we have
         \begin{linenomath}
         \begin{equation}
         \label{BDformula}
         F(u) = F_d(u) + \int_{\o}u^2k(dx) + \int_{(\o\times\o)\setminus{\rm diag}} (u(x)-u(y))^2j(dx,dy),
         \end{equation}
         \end{linenomath}
    where  $F_d$ is called the diffusion part of $F$, $k$ is a positive Radon measure on $\o$, called the killing measure, and $j$ is a positive Radon measure on $(\o\times\o)\setminus{\rm diag}$, called the jumping measure.  Recall that a Dirichlet form $F$ is a closed form which satisfies the Markovian property, {\it i.e.} for any  contraction $T:\R\to\R$, such that 
        \begin{equation*}
        T(0)=0, \qquad\mbox{and}\qquad \forall x,y\in\R, \hspace{0.3cm} |T(x)-T(y)|\leq |x-y|,
        \end{equation*}
     we have $F\circ T\leq F$. A $\Gamma$-limit form obtained with the $L^2(\o)$-weak topology does not \textit{a priori} satisfy the Markovian property, since the $L^2(\o)$-weak convergence does not commute with all contractions $T$. An example of a sequence of Markovian forms whose $\Gamma$-limit for the $L^2(\o)$-weak topology does not satisfy the Markovian property is provided in {\rm \cite[Theorem 3.1]{B03}}.
     Hence, the representation formula \eqref{BDformula} does not hold in general when the $L^2(\o)$-strong topology is replaced by the $L^2(\o)$-weak topology. 
    \par In the present context,  we do not know if the $\Gamma$-limit $\mathscr{F}$  \eqref{funccontex} is a Dirichlet form since the presence of the convolution term makes difficult to prove the Markovian property.
    \end{remark}

\appendix \section{Homogenized formula for a rank-one laminate}
%The so-called lamination formula provides the homogenized properties of a two-phase rank-one laminate with non-degenerate phases. The explicit expression of $A^\ast$ depends only on the matrices $A_1$ and $A_2$, the volume fraction $\theta$ and the laminate direction (see  \cite[Lemma 1.3.32]{A12} and \cite{B94} for the laminate homogenization formula in the case of reiterated laminate structure). 
We are going to give an explicit expression of the homogenized matrix $A^\ast$ defined by \eqref{hommat}, which extends the rank-one laminate formula %(see  \cite[Lemma 1.3.32]{A12} and \cite{B94} for the laminate homogenization formula in the case of reiterated laminate structure) 
in the case of a rank-one laminates  with degenerate phases.  We will recover directly from this expression the positive definiteness of $A^*$ for the class of rank-one laminates introduced in Section 3. Indeed, by virtue of Theorem 2.1 the positive definiteness of $A^*$ also follows from assumption (H2) which is established in Proposition 3.1 and Proposition 3.2.\\
%We will also study the positive definiteness of the laminate homogenized matrix $A^\ast$ for the class of the laminate materials introduced in Section $3$.\\
Set
           \begin{linenomath}
           \begin{equation}
           \label{a}
            a:= (1-\theta) A_1 e_1\cdot e_1 + \theta A_2e_1\cdot e_1,
            \end{equation}
            \end{linenomath}
with $\theta\in (0,1)$ being the volume fraction of phase $Z_1$.
       \begin{prop}
       \label{prop:Appendix}
       Let $A_1$ and $A_2$ be two symmetric and non-negative matrices of $\R^{d\times d}$, $d\geq 2$. If  $a$ given by \eqref{a} is positive, the homogenized matrix $A^\ast$ is given by    
              \begin{equation}
              \label{hommatrlamform}
               A^\ast= \theta A_1+(1-\theta)A_2 -\frac{\theta(1-\theta)}{a} (A_2-A_1)e_1\otimes (A_2-A_1)e_1.
               \end{equation}
       If $a=0$, the homogenized matrix $A^\ast$ is the arithmetic average of the matrices $A_1$ and $A_2$, \textit{i.e.} 
                  \begin{equation}
                  \label{hommatrlamformazero}
                  A^\ast = \theta A_1+(1-\theta)A_2.
                  \end{equation}
        Furthermore, if one of the following conditions is satisfied:
          \begin{itemize}
          \item[i) ]  in two dimensions, $a>0$ and the matrices $A_1$ and $A_2$ are given by \eqref{phasesdimtwo} %such that
          %       \begin{equation}
          %       \label{matrixappedtwo}
          %       A_1=\xi\otimes\xi, \qquad A_2\hspace{0.2cm}\mbox{positive definite}
           %      \end{equation}
          with  $\xi\cdot e_1\neq 0$,   
          %\item[ii) ] in two dimensions, $a=0$ and $A_2$ is positive definite;
          \item[ii) ]  in three dimensions, $a>0$, the matrices $A_1$ and $A_2$ are given by \eqref{kernranktwo} and the vectors $\{e_1, \eta_1, \eta_2\}$ are independent in $\R^3$,
          \end{itemize}
        then $A^\ast$ is positive definite.
       \end{prop}
   \begin{remark}
   The condition $a>0$ agrees with the $\Gamma$-convergence results of Propositions \ref{propd2}  and \ref{thmcasematrixdege}. In the two-dimensional framework, the degenerate case $a=0$ does not agree with Propositions \ref{propd2}. Indeed, $a=0$ implies that $A_1e_1\cdot e_1= A_2 e_1\cdot e_1=0$ in contradiction to positive definiteness of $A_2$. Similar in the three-dimensional setting, where the independence of  $\{e_1, \eta_1, \eta_2 \}$ is not compatible with $a=0$. Indeed, $a=0$ implies that $A_i e_1 =A_i \eta_i=0$, for $i=1,2$, which contradicts the fact that $A_1$ and $A_2$ have rank two.
    \end{remark}
     \begin{proof}
     Assume that $a>0$. In view of the convergence \eqref{AdeltatendA}, we already know that
          \begin{equation}
          \label{convmatrApp}
          \lim_{\delta\to 0} A^\ast_\delta = A^\ast,
          \end{equation}
     where, for $\delta>0$, $A^\ast_\delta$ is the homogenized matrix associated to  conductivity matrix $A_\delta$ given by 
          \begin{linenomath}
          \begin{equation*}
          A_\delta(y_1) = \chi(y_1) A_1^\delta + (1-\chi(y_1)) A_2 ^\delta\qquad\mbox{for}\hspace{0.2cm} y_1\in\R,
          \end{equation*}
          \end{linenomath}
     with $A_i^\delta = A_i +\delta I_d$. Since $A_1$ and $A_2$ are non-negative matrices, $A_\delta$ is positive definite and thus the homogenized matrix $A^\ast_\delta$ is given by the lamination formula (see \cite{T85} and also \cite[Lemma 1.3.32]{A12})
         \begin{equation}
         \label{lamformu}
         A^\ast_\delta= \theta A_1^\delta+(1-\theta)A_2^\delta -\frac{\theta(1-\theta)}{(1-\theta)A_1^\delta e_1\cdot e_1 + \theta A_2^\delta e_1\cdot e_1 } (A_2^\delta-A_1^\delta)e_1\otimes (A_2^\delta-A_1^\delta)e_1.
         \end{equation}   
      If $a>0$, we easily infer from the convergence \eqref{convmatrApp} combined with the lamination formula \eqref{lamformu} the expression \eqref{hommatrlamform}  for $A^\ast$.
     \par Let us prove that $A^\ast x\cdot x\geq 0$ for any $x\in\rd$. From the Cauchy-Schwarz inequality, we deduce that 
                     \begin{linenomath}
                    \begin{align}
                    \label{CSine}
                    |(A_2-A_1)e_1\cdot x| &\leq |A_2e_1\cdot x|+ |A_1e_1\cdot x|\notag\\
                    &\leq (A_2e_1\cdot e_1)^{1/2}(A_2x\cdot x)^{1/2} + (A_1e_1\cdot e_1)^{1/2}(A_1x\cdot x)^{1/2}.
                    \end{align}
                    \end{linenomath}
      This combined with  the definition \eqref{hommatrlamform} of $A^\ast$ implies that, for any $x\in\rd$,
                  \begin{linenomath}
                  \begin{align}
                  A^\ast x\cdot x&=\theta (A_1 x\cdot x) + (1-\theta)(A_2 x\cdot x) - \theta(1-\theta)a^{-1}\l|(A_2-A_1)e_1\cdot x\r|^2\notag\\
                 &\geq \theta (A_1x\cdot x) +(1-\theta) (A_2 x\cdot x)\notag\\ &\quad-\theta(1-\theta)a^{-1}[(A_2e_1\cdot e_1)^{1/2} (A_2x\cdot x)^{1/2}+ (A_1e_1\cdot e_1)^{1/2} (A_1x\cdot x)^{1/2}  ]^2\notag\\
                 &=  a^{-1} [a \theta (A_1x\cdot x) +a (1-\theta) (A_2 x\cdot x) -\theta(1-\theta)(A_2e_1\cdot e_1)( A_2x\cdot x)\notag\\
                 &\hspace{0.9cm}  -\theta(1-\theta)(A_1e_1\cdot e_1)(A_1x\cdot x)- 2\theta(1-\theta)(A_2e_1\cdot e_1)^{1/2}( A_2x\cdot x)^{1/2}(A_1e_1\cdot e_1)^{1/2}(A_1x\cdot x)^{1/2}].\label{eii}
                  %&=a^{-1}[\theta (A_1x\cdot x)^{1/2}(A_2 e_1\cdot e_1)^{1/2} -(1-\theta) (A_2x\cdot x)^{1/2}(A_1 e_1\cdot e_1)^{1/2}  ]^2\geq 0, 
                  \end{align}
                  \end{linenomath}
      In view of definition \eqref{a} of $a$, we have that
         \begin{linenomath}
         \begin{align}
         a \theta( A_1x\cdot x) +a (1-\theta)( A_2 x\cdot x)&= \theta (1-\theta)(A_1 e_1\cdot e_1)(A_1 x\cdot x) + \theta^2(A_2 e_1\cdot e_1)(A_1 x\cdot x)\notag\\
                 &\quad+ (1-\theta)^2
                 (A_1 e_1\cdot e_1)(A_2 x\cdot x) + \theta(1-\theta)(A_2 e_1\cdot e_1)(A_2 x\cdot x).\notag  
         \end{align}
               \end{linenomath}
      Plugging this equality in \eqref{eii}, we deduce that
             \begin{linenomath}
             \begin{align}
             A^\ast x\cdot x &\geq a^{-1}[\theta^2(A_2 e_1\cdot e_1)(A_1 x\cdot x)+ (1-\theta)^2(A_1 e_1\cdot e_1)(A_2 x\cdot x)\notag\\
             &\hspace{1cm} -2\theta(1-\theta)(A_2e_1\cdot e_1)^{1/2}(A_1x\cdot x)^{1/2}(A_1e_1\cdot e_1)^{1/2}( A_2x\cdot x)^{1/2}]\notag\\
             &=a^{-1}[\theta (A_2 e_1\cdot e_1)^{1/2}(A_1x\cdot x)^{1/2} -(1-\theta) (A_1 e_1\cdot e_1)^{1/2}(A_2x\cdot x)^{1/2}  ]^2\geq 0,\label{eiii}
             \end{align}
             \end{linenomath}
      which proves that $A^\ast$ is a non-negative definite matrix.
      \par Now, assume $a=0$. Since $A_1$ and $A_2$ are non-negative matrices, the condition $a=0$ implies  $A_1e_1\cdot e_1 = A_2e_1\cdot e_1=0$ or equivalently  $A_1e_1=A_2e_1=0$. Hence,
             \begin{linenomath}
             \begin{equation*}
             (A_2^\delta -A_1^\delta)e_1 = (A_2-A_1)e_1 =0,
             \end{equation*}
             \end{linenomath}
       which implies that the lamination formula \eqref{lamformu} becomes
         \begin{linenomath}
         \begin{equation*}
         A^\ast_\delta = \theta A_1^\delta + (1-\theta)A_2^\delta.
         \end{equation*}
         \end{linenomath}
     This combined with the convergence \eqref{convmatrApp} yields to the expression
     \eqref{hommatrlamformazero} for $A^\ast$. 
\par To conclude the proof, it remains to prove the positive definiteness of $A^\ast$ under the above conditions i) and ii). 

\smallskip\noindent
Case (i): $d=2$, $a>0$ and $A_1, A_2$ given by \eqref{phasesdimtwo}.\\ Assume $A^* x\cdot x=0$. Then, the inequality \eqref{eiii} is an equality, which yields in turn  equalities in \eqref{CSine}. In particular, we have
        \begin{linenomath}
        \begin{equation}
        \label{A2e1xposdef}
        |A_2e_1\cdot x| = (A_2e_1\cdot e_1)^{1/2}(A_2x\cdot x)^{1/2} = \|A^{1/2}_2e_1\|\|A^{1/2}_2x\|.
        \end{equation}
        \end{linenomath} 
%Since $A_2$ is positive definite, $\|A^{1/2}_2x\|$ defines actually a norm in $\rd$.
Recall that the Cauchy-Schwarz inequality is an equality if and only if one of vectors is a scalar multiple of the other.
This combined with \eqref{A2e1xposdef} leads to   $A^{1/2}_2x=\alpha A^{1/2}_2e_1$ for some $\alpha\in\R$, so that,  since $A_2$ is positive definite or equivalently $A^{1/2}_2$, we have 
   \begin{equation}
   \label{xe1}
   x=\alpha e_1 \qquad\mbox{for some}\hspace{0.1cm}\alpha\in\R.
   \end{equation}
From the definition \eqref{hommatrlamform} of $A^\ast$ and due to the assumption  $\xi\cdot e_1\neq 0$, we get
   \begin{equation}
   \label{Ae1e1}
   A^\ast e_1\cdot e_1= \frac{1}{a}(A_2 e_1\cdot e_1) (\xi\cdot e_1)^2 >0.
   \end{equation}
Recall that $A^\ast x\cdot x=0$.  This combined with  \eqref{xe1} and \eqref{Ae1e1} implies that $x=0$, which proves that $A^\ast$ is positive definite.  

%\smallskip\noindent
%Case(ii): $d=2$  and $a=0$. \\In  light of  definition \eqref{hommatrlamformazero},  the positive definiteness of $A^\ast$ follows immediately from that of $A_2$. 

\smallskip\noindent
Case (ii): $d=3$, $a>0$ and $A_1, A_2$ given by \eqref{kernranktwo}. \\Assume that $A^\ast x\cdot x=0$. As in Case (i), we have  equalities in \eqref{CSine}. In other words,
    \begin{linenomath}
   	\begin{align}
   	|A_1e_1\cdot x| &= (A_1e_1\cdot e_1)^{1/2}(A_1x\cdot x)^{1/2},\label{A1e1x}\\
   	|A_2e_1\cdot x| &= (A_2e_1\cdot e_1)^{1/2}(A_2x\cdot x)^{1/2}.\label{A2e1x}
   	\end{align}
   	\end{linenomath}
Let $p_i(t)$ be the non-negative polynomials of degree $2$ defined by 
     \begin{linenomath}
     \begin{equation*}
     p_i(t) := A_i(x+te_1)\cdot (x+te_1) \qquad\mbox{for}\hspace{0.1cm} i=1,2.
     \end{equation*}
     \end{linenomath}
In view of \eqref{A1e1x}, the discriminant of $p_1(t)$ is zero, so that there exists $t_1\in\R$ such that 
     \begin{equation}
     \label{poly}
     p_1(t_1) = A_1(x+t_1e_1)\cdot (x+t_1e_1)=0.
     \end{equation}
Recall that $\ker (A_1)={\rm Span}(\eta_1)$. Since $A_1$ is non-negative matrix, we deduce from \eqref{poly} that $x+t_1e_1$ belongs to $\ker (A_1)$, so that
     \begin{equation}
     \label{spane1eta1}
     x\in{\rm Span}(e_1, \eta_1).
     \end{equation}
Similarly,  recalling that $\ker (A_2)={\rm Span}(\eta_2)$ and  using \eqref{A2e1x}, we have
     \begin{equation}
     \label{spane1eta2}
     x\in{\rm Span}(e_1, \eta_2).
     \end{equation}
Since the vectors $\{e_1, \eta_1, \eta_2\}$ are independent in $\R^3$, \eqref{spane1eta1} and \eqref{spane1eta2} imply that 
      \begin{linenomath}
      \begin{equation*}
      x=\alpha e_1 \qquad\mbox{for some}\hspace{0.2cm}\alpha\in\R.
      \end{equation*}
      \end{linenomath}
In light of definition \eqref{hommatrlamform} of $A^\ast$, we have 
    \begin{linenomath}
    \begin{equation*}
    A^\ast e_1\cdot e_1 = \frac{1}{a} (A_1e_1\cdot e_1)(A_2e_1\cdot e_1)>0,
    \end{equation*}  
    \end{linenomath} 
which implies that $x=0$, since $A^\ast x\cdot x=0$.  This establishes that $A^\ast$ is positive definite and concludes the proof.
   \end{proof}
Note that when $d=2$ and $a>0$ the assumption $\xi\cdot e_1\neq 0$ is essential to obtain that $A^\ast$ is positive definite. Otherwise, the homogenized matrix $A^\ast$ is just non-negative definite as shown by the following counter-example. Let $A_1$ and $A_2$ be symmetric and non-negative matrices of $\R^{2\times 2}$ defined by
      \begin{linenomath}
      \begin{equation*}
      A_1=e_2\otimes e_2 \quad\mbox{and}\quad A_2 = I_2.
      \end{equation*}
      \end{linenomath}
Then, it is easy to check that $a=\theta>0$ and $A^\ast e_1\cdot e_1=0$.

\subsection*{Acknowledgments}
This problem was pointed out to me by  Marc Briane during my stay at  Institut National des
Sciences Appliquées de Rennes.  I thank him for the countless fruitful discussions. My thanks also extend  to Valeria Chiad\`{o} Piat for useful remarks.
The author is also a member of the INdAM-GNAMPA project  "Analisi variazionale di modelli non-locali nelle scienze applicate".

\bibliographystyle{plain}
%\bibliography{References}

\end{document}